\documentclass{amsart}

\usepackage{amssymb}
\usepackage{amsmath,amsthm}
\usepackage{amsthm}
\usepackage{color}
\usepackage{hyperref,url}

\calclayout

\hypersetup{
    colorlinks=true,
    linkcolor=blue,
    filecolor=blue,
    urlcolor=black,
    citecolor=blue,
}

\numberwithin{equation}{section}

\definecolor{plum}{rgb}{.5,0,1}

\newtheorem{theorem}{Theorem}
\newtheorem{proposition}{Proposition}
\newtheorem{corollary}{Corollary}

\theoremstyle{plain}
\newtheorem{lemma}{Lemma}

\theoremstyle{definition}

\newtheorem{definition}{Definition}

\newtheorem{remark}{Remark}

\newcommand{\Z}{\mathbb{Z}}

\newcommand{\R}{\mathbb{R}}
\newcommand{\C}{\mathbb{C}}

\newcommand{\e}{\epsilon}

\newcommand{\supp}{\mathrm{supp}\, }

\DeclareFontFamily{U}{mathx}{\hyphenchar\font45}
\DeclareFontShape{U}{mathx}{m}{n}{
      <5> <6> <7> <8> <9> <10>
      <10.95> <12> <14.4> <17.28> <20.74> <24.88>
      mathx10
      }{}
\DeclareSymbolFont{mathx}{U}{mathx}{m}{n}
\DeclareFontSubstitution{U}{mathx}{m}{n}
\DeclareMathAccent{\widecheck}{0}{mathx}{"71}
\DeclareMathAccent{\wideparen}{0}{mathx}{"75}

\title{Null Structures and Degenerate Dispersion Relations in Two
  Space Dimensions} \date{} \author{Yuqiu Fu} \author{Daniel Tataru}
\thanks{The second author was partially supported by the NSF grant
  DMS-1266182 as well as by a Simons Investigator grant from the
  Simons Foundation.}  \address{Department of Mathematics, 970 Evans
  Hall, Berkeley, CA 94720}
\email{\href{mailto:yuqiufu@outlook.com}{yuqiufu@outlook.com} \\
  \href{mailto:tataru@math.berkeley.edu}{tataru@math.berkeley.edu}}

\begin{document}

\begin{abstract}
  For a dispersive PDE, the degeneracy of its dispersion relation will
  deteriorate dispersion of waves, and strengthen nonlinear
  effects. Such negative effects can sometimes be mitigated by some
  null structure in the nonlinearity.

  Motivated by water-wave problems, in this paper we consider a class
  of nonlinear dispersive PDEs in 2D with cubic nonlinearities, whose
  dispersion relations are radial and have vanishing Guassian
  curvature on a circle.  For such a model we identify certain null
  structures for the cubic nonlinearity, which suffice in order to
  guarantee global scattering solutions for the small data problem.
  Our null structures in the power-type nonlinearity are weak, and
  only eliminate the worst nonlinear interaction. Such null structures
  arise naturally in some water-wave problems.
\end{abstract}

\maketitle

\section{Introduction}\label{introsection}

We consider the following model Cauchy problem
\begin{equation}\label{modeleqn}
  \left\{
    \begin{array}{l}
      \partial_t u-ih(D)u=A(D)(|P_{\leq M}u|^2P_{\leq M}u)\\
      \\
      u(0,x)=u_0\in L^2(\R^2).
    \end{array}
  \right.
\end{equation}
with a cubic nonlinearity.

Here $h(\xi)$ is a radial dispersion relation on $\R^2$ which is
degenerate on the unit circle, i.e. its Hessian is degenerate there.
$P_{\leq M}$ is a cutoff in the frequency space selecting an annulus
near the unit circle; this is where the strongest nonlinear interactions
are occuring.

The interesting object here is the multiplier $A$, whose symbol
$A(\xi)$ vanishes to some order on the same circle.  We observe that
without $A(D)$, the worst nonlinear interaction of waves in the
power-type nonlinearity occurs when waves have equal frequencies on
the unit circle. Therefore intuitively the Fourier multiplier $A(D)$
provides a null structure in (\ref{modeleqn}) by eliminating the worst
nonlinear interaction.

We will show that such types of null structures are indeed effective
in controlling the long time dynamics, provided that the order of
vanishing of $A(\xi)$ on the unit circle is no less than half of that
of the order of degeneracy there for the Hessian of $H$.  Under these
conditions, we will prove that small data leads to global scattering
solutions.

The placement of the multiplier $A(D)$ in the equation is less important. The same result
holds if $A(D)$ applies instead to any of the factors in the cubic
nonlinearity, e.g. as in $P_{\leq M}(A(D)u P_{\leq M} \overline{u} P_{\leq M} u)$, etc.

We now describe our model in more detail. First, we suppose there
exists $\delta\in (0,1)$ and a smooth function
$\gamma:(1-\delta,1+\delta)\rightarrow \R$ such that
$h(\xi)=\gamma(|\xi|).$ We make the following assumptions on $\gamma$:

\begin{enumerate}\label{assumptiongamma}
\item Transversality: $|\gamma'(r)| \approx 1$ for every $r \in (1-\delta,1+\delta).$

\item Degeneracy: There exists $\beta \in \Z_+$  such that
  $|\gamma''(r)| \approx |r-1|^\beta$ for every
  $r\in (1-\delta,1+\delta).$
\end{enumerate}

The first condition guarantees the transversality of waves with
angularly separated frequencies, and the second condition gives a
finite order $\beta$ of dispersion degeneracy on the unit circle.
Under these conditions, the worst cubic nonlinear interactions
will turn out to occur between equal frequencies on the unit circle.

The Fourier multiplier $P_{\leq M}$  is a smooth cutoff in the frequency
space selecting the region $|\xi|\in (1-2^{M+1},1+2^{M+1})$, where
the integer $M<-100$ is chosen such that $2^{M+4}<\delta$.
This is  introduced so that we restrict our attention to the
frequency region where the most interesting part of nonlinear
interactions occur.

Last but not least, the  Fourier multiplier $A(D)$ is given by a symbol $A(\xi)$ with the
property that
\begin{equation}\label{assumptionA}
  |A(\xi)| \lesssim  \left||\xi|-1\right|^{\beta/2}
\end{equation}
for $|\xi|\in (1-\delta,1+\delta),$ and
\begin{equation}\label{4}
  \supp A(\xi)\subset \{\xi\in \R^2: 1-2^{M-1} < |\xi|< 1+2^{M-1}\}.
\end{equation}


The motivation for considering the Cauchy problem (\ref{modeleqn})
comes from the study of long-term dynamics for water-wave systems.
For an incompressible, inviscid and irrotational fluid occupying a
time-dependent domain
$$\Omega_t=\{(x,y)\in \R^2\times \R: -H_0\leq y\leq h(x,t)\}$$
for some function $h(x,t)$ and for $t$ in some interval $I_t\subset
\R,$ the water-wave problem can be reduced (see for example
\cite{lannes2013water}) to the following formulation for unknowns
$h,\phi: \R^2_x \times I_t \rightarrow \R: $
\begin{equation}\label{waterwaveeqn}
  \begin{cases}
    \partial_t h=G(h)\phi \\
    \partial_t \phi =-gh + \dfrac{(G(h)\phi+\nabla h\cdot \nabla
      \phi)^2}{2(1+|\nabla h|^2)} -\dfrac{1}{2}|\nabla \phi|^2+\sigma
    \mathrm{div}\left( \dfrac{\nabla h}{(1+|\nabla h|^2)^{1/2}}
    \right),
  \end{cases}
\end{equation}
where $\sigma\geq 0$ is the surface tension coefficient, $g$ is the
gravitational constant, and
$$G(h):=\sqrt{1+|\nabla h|^2} \mathcal{N}(h).$$
Here $\mathcal{N}(h)$ is the Dirichlet-Neumann map associated to the
domain $\Omega_t,$ and $\phi(x,t)$ is the restriction of the velocity
potential to the boundary surface $y=h(x,t).$

The dispersion relation for the linearized equation of
(\ref{waterwaveeqn}) around the zero solution is given by
$$\Lambda(\xi)=\pm \sqrt{|\xi|(g+\sigma |\xi|^2)\tanh (H_0 |\xi|)}.$$
For simplicity of computation we consider the infinite depth case
$H_0\rightarrow \infty,$ in which case the dispersion relation becomes
$\Lambda=\pm \sqrt{g|\xi|(1+\sigma |\xi|^2/g)}.$ If we are considering
gravity-capillary water waves, that is $\sigma, g>0$, then
away from the origin, $\det( \nabla^2 \Lambda)$ vanishes exactly on a
circle, which implies that we will not have optimal dispersion of
waves. This motivates (and is an example of) our general dispersion
relations considered above.  Our nonlinearity has been chosen to be of
cubic type, as for this dispersion relation, nonlinear interactions in
quadratic nonlinearities are either non-resonant or
transversal. Because of this, one expects that the long-time dynamics
are in fact primarily  governed by cubic nonlinear interactions.

Returning to our model problem, we begin by observing that,
because of the frequency localization in the nonlinearity, by
  H\"{o}lder's inequality and Bernstein's inequality we have
\[
\|A(D)(P_{\leq M} u_1 \overline{P_{\leq M} u_2} P_{\leq M} u_3)\|_{L^2_x}\lesssim
 \|u_1\|_{L^2_x}\|u_2\|_{L^2_x}\|u_3\|_{L^2_x}.
\]
From this estimate one can easily establish local wellposedness of
(\ref{modeleqn}) in $L^2(\R^2)$ using a standard fixed point
argument (see Section \ref{proofmainsection}):

\begin{theorem} [Local well-posedness]
For every $R>0,$ there exists $T=T(R)>0$ such that for every $u_0\in L^2(\R^2)$
with $\|u_0\|_{L^2_x}\leq R,$ there exists a unique strong solution
$u\in C^0([0,T),L^2_x)$ to the Cauchy problem (\ref{modeleqn}), and
the solution map $$B_{L^2_x}(0,R)\rightarrow C^0([0,T),L^2_x),\,
u_0\mapsto u$$ is Lipschitz continuous.

\end{theorem}


The local well-posedness result above does not use at all the dispersive
properties of the equation. This, however, becomes crucial if one
consider the global in time well-posedness question. This is the
main goal of this paper. Our result is as follows:

\begin{theorem}[Global wellposedness for small
  data]\label{smallglobalwellposed}
  There exists $\epsilon_0>0$ such that for every $u_0\in L^2(\R^2)$
  with $\|u_0\|_{L^2_x}\leq \epsilon_0,$ there
  exists a unique strong global solution $u\in X^0([0,\infty)) \cap
  C^0([0,\infty); L^2_x)$ to the Cauchy problem (\ref{modeleqn}), and the
  solution map
   $$B_{L^2_x}(0,\epsilon_0)\rightarrow X^0([0,\infty)) \cap C^0([0,\infty),L^2_x),\, u_0\mapsto u$$
   is Lipschitz continuous.

Furthermore, the solutions are
scattering, i.e. for each small data $u_0 \in L^2(\R^2)$ there exist
$u_+ \in L^2(\R^2)$, small, with Lipschitz dependence on $u_0$, so that
\[
\lim_{t \to  \infty} \left( u(t) - e^{it h(D)} u_{+} \right) = 0 \qquad \text{in } L^2(\R^2).
\]
 \end{theorem}

\begin{remark}
  The above wellposedness and scattering result also holds for small solutions that are backward in time. Indeed if $u(t,x)$ is a solution to (\ref{modeleqn}), then $\overline{u(-t,x)}$ is a solution to (\ref{modeleqn}) with $A(D)$ replaced by $A(-D),$ and initial data $\overline{u_0}.$
\end{remark}

The function space $X^0$ captures the
dispersive properties of the solutions, and we have the embedding
\[
X^0([0,\infty))\subset L^\infty_t L^2_x([0,\infty)\times \R^2).
\]
This is introduced in Section \ref{functionspacesection} and is defined
using $U^2$ type spaces associated to the corresponding linear flow.

A key role in the proof of our result is played by  localized Strichartz estimates and bilinear
$L^2$ estimates for solutions to the linear  homogeneous flow.  These  are  derived in Section
\ref{estimatessection}; we hope they will also be of independent interest due to the optimal
 treatment of the degeneracy.

 The linear and bilinear estimates are transferred to the $X^0$ space
 in Section~\ref{nonlinearestimate}. This in turn allows us to
 conclude the proof of our small data result in
 Section~\ref{proofmainsection}.

%

 \bigbreak
 \noindent {\bf Notations:}
In what follows  $\widehat{f}$ will denote the spatial Fourier transform of the function $f$,
\[
\widehat{f}(t,\xi)=\frac{1}{2\pi} \int_{\R^2}f(t,x)e^{-ix\cdot \xi}dx,
\]
and $\widecheck{g}$ will denote the inverse Fourier transform, such that
$(\widehat{f})\,\widecheck{}=f$. We denote the standard inner product on $L^2$ by
 $\langle\, ,\,
\rangle,$ that is
\[
\langle f,g \rangle =\int f\overline{g}.
\]

If $X,Y$ are two subsets of $\R^n,$ then $X\Subset Y$ means that $X$ is
contained in a compact subset of $Y$, and $X+Y$ denotes the set
$\{x+y:x\in X, y\in Y\}$.  $1_X$ will denote the indicator function of
the set $X$. The diameter of $X$ is defined as
\[
\mathrm{diam}\, X =\sup_{x,x'\in X} |x-x'|.
\]

We let $A\lesssim B$ denote the statement there exists a  constant $c>0$
such that $A\leq cB,$ and let $A\gtrsim B$ be the statement $B\lesssim
A.$ We also let $A\sim B$ be the statement $A\lesssim B$ and
$B\lesssim A.$ We need to be careful what those implicit constants
depend on. Unless otherwise specified, they depend only on $\gamma$
(and related parameters appeared in this section), and $\chi$ which is
chosen in Section \ref{functionspacesection}. We call those constants
admissible. We will also sometimes describe the dependence explicitly.

\section{Function Spaces}\label{functionspacesection}

To prove Theorem \ref{smallglobalwellposed}, we will use a
perturbative argument. Our function spaces will be of $U^p,V^p$ type, see
\cite{kochtataru2005upvp}. Since most results in this section
concerning $U^p,V^p$ type spaces have already been well developed, we
will usually only provide references but no proofs for them. A
detailed exposition of those spaces can be found in
\cite{kochtataruvisan2014dispersive}.  Unless otherwise stated, we
assume $p\in (1,\infty).$

Let $\mathcal{Z}$ be the collection of finite partitions of the real
line
$$\mathcal{Z}=\{(t_0,\ldots, t_K): -\infty<t_0< \cdots <t_K=\infty\}.$$

\begin{definition}{\label{Up}}
  We call the function $a:\R\rightarrow L^2(\R^2)$ a $U^p-$atom if
  $$a=\sum_{i=1}^{K}1_{[t_{i-1}, t_{i})}\phi_{i-1}$$
  for some partition $(t_j)\in \mathcal{Z},$ and some $\phi_i\in
  L^2(\R^2)$ satisfying $\sum_{i=0}^{K-1}\|\phi_i\|^p_{L^2(\R^2)}=1.$
  We then define $U^p(\R,L^2(\R^2))$ to be the space of functions
  $u:\R\rightarrow L^2(\R^2)$ such that
  $$\|u\|_{U^p}:=\inf\left\{\sum_{j=1}^{\infty}|\lambda_j|: u=\sum_{j=1}^{\infty}\lambda_ja_j, \, \lambda_j\in \C, \, a_j \text{ are } U^p\text{-atoms}\right\}<\infty,$$
  with norm $\|\cdot\|_{U^p}.$ Here we adopt the convention that $\inf
  \emptyset=\infty.$
\end{definition}

\begin{definition}
  $V^p(\R,L^2(\R^2))$ is defined to be the space of functions
  $v:\R\rightarrow L^2(\R^2)$ for which
  $$\|v\|_{V^p}:=\sup_{(t_k)_k\in \mathcal{Z}}\left( \sum_{k=1}^{K} \|v(t_k)-v(t_{k-1})\|_{L^2(\R^2)}^p \right)^{1/p}<\infty,$$
  where we use the convention that $v(t_K)=0$ if $t_K=\infty.$ We let
  $V_{rc}^p(\R,L^2(\R^2))$ be the subspace consisting of
  right-continuous functions $v:\R\rightarrow L^2(\R^2)$ in $V^p$ such
  that $\lim_{t\rightarrow -\infty} v(t)=0.$
\end{definition}

\begin{remark} \hfil
  \begin{itemize}
  \item $U^p(\R,H),V^p(\R,H)$ spaces can be similarly defined for
    functions from the real line to any complex Hilbert space $H,$ but
    for our application in this paper we will only consider functions
    with $H=L^2(\R^2).$ We will usually omit the space $L^2(\R^2)$ and
    domain $\R$ in the notation, writing $U^p,V^p$ instead.
  \item For $1\leq p<q<\infty,$ we have the embeddings
    \begin{equation}\label{embeddingUpVp}
      U^p\hookrightarrow V^p_{rc}\hookrightarrow U^q\hookrightarrow L^\infty_t L^2_x,
    \end{equation}
    a proof of which can be found in
    \cite{kochtataruvisan2014dispersive}.
  \item $U^p,V^p$ spaces are Banach spaces. $V^p_{rc}$ is a closed
    subspace of $V^p,$ which can be seen immediately from
    definitions and the embedding $V^p_{rc}\hookrightarrow L^\infty_t
    L^2_x.$
  \item $U^2,V^2$ spaces behave well under time truncation and
    frequency truncation. More precisely we have
    \begin{equation}\label{-99}
      \|1_{[a,b)}u\|_{U^2}\lesssim \|u\|_{U^2}
    \end{equation}
        $$\|P_ku\|_{U^2}\lesssim \|u\|_{U^2},$$
        where the implicit constants are independent of $[a,b)$ or
        $k$. For the definition of $P_k$ see Definition
        \ref{projector} below. The same inequalities hold if $U^2$ is
        replaced by $V^2.$ These estimates can be easily checked using
        definitions.
      \end{itemize}
    \end{remark}

\begin{definition}
  We let $U^p_{h(D)}$ be the space
  $$U^p_{h(D)}:=\left\{e^{ith(D)}u: u\in U^p\right\}$$
  with the norm
  $$\|u\|_{U^p_{h(D)}}=\|e^{-ith(D)}u\|_{U^p}.$$
  We similarly define spaces $V^p_{h(D)},$ $V^p_{h(D),rc}$
  corresponding to $V^p,$ $V^p_{rc}$ respectively.
\end{definition}
Those spaces are Banach spaces since  the original spaces are, and
$e^{ith(D)}$ is unitary on $L^2$.

As can be seen in the next few sections, our strategy is to perform a
dyadic decomposition around the singular set $|\xi|=1.$ To make this
precise, we let $\chi:\R \rightarrow [0,1]$ be a smooth function
satisfying
$$\chi(r)=\begin{cases}
  1 & \mbox{if } |r|\leq 1/2, \\
  0 & \mbox{if } |r|\geq 3/4.
\end{cases}$$

\begin{definition}[Dyadic decomposition around the unit circle]\label{projector}
  For $k< 0,$ we let $P_{\leq k}$ be the Fourier multipliers given by
  symbols
  \begin{equation}\label{projectoreqn}
    P_{\leq k}(\xi)=\chi\left(2^{-k}(|\xi|-1))\right).
  \end{equation}
  We then define $P_{k}=P_{\leq k}-P_{\leq k-1}.$ Also expressions
  like $P_{k_1< \cdot \leq k_2}$ are defined in the obvious way. We
  let $P_0$ be the Fourier multiplier with symbol
  $$P_0(\xi):=1-P_{<0}(\xi).$$
\end{definition}

\begin{remark}
  We observe that $P_k(\xi)$ is supported on the disjoint union of two
  annuli
\[
\{ 1-2^{k}< |\xi|< 1-2^{k-2}\} \cup \{ 1+2^{k-2}< |\xi|< 1+2^k\}.
\]
We write $P_k(\xi)=P_k^+(\xi)+P_k^-(\xi)$ where $P_k^+$ and $P_k^-$
are supported on $1+2^{k-2}\leq |\xi|\leq 1+2^k$ and $1-2^{k}\leq
|\xi|\leq 1-2^{k-2}$ respectively. We can of course do the same
decomposition to $P_{k_1\leq \cdot \leq k_2}.$ It is sometimes
convenient for us to consider the $+$ part and the $-$ part
separately, and usually the same argument works for both.
\end{remark}

We also observe that projections $P_k$ are almost orthogonal, and
therefore we have
\begin{equation}\label{almostalthogonal}
  \|u\|_{L^2(\R^2)} \sim \left( \sum_{k=-\infty}^{0}\|P_ku\|_{L^2(\R^2)}^2 \right)^{1/2}.
\end{equation}

Now we can introduce the function spaces we will use in the fixed
point argument.

\begin{definition}\label{Y0Z0space}
  Let $Y^0$ be the space of functions $v:\R\rightarrow L^2(\R^2)$ such
  that $P_kv\in V^2_{h(D),rc}$ and
  $$\|v\|_{Y^0}:=\left( \sum_{k=-\infty}^{0}\|P_k v\|_{V^2_{h(D)}}^2 \right)^{1/2}<\infty,$$
  endowed with the above norm.  Let $X^0$ be the space of functions
  $u:\R\rightarrow L^2(\R^2)$ for which $P_kv\in U^2_{h(D)}$ and
  $$\|u\|_{X^0}:=\left( \sum_{k=-\infty}^{0}\|P_k v\|_{U^2_{h(D)}}^2 \right)^{1/2}<\infty,$$
  endowed with the above norm.
\end{definition}

\begin{remark}\label{X0Y0embedLinfL2}
  We first observe that $\|\cdot\|_{L^\infty_t L^2_x} \lesssim
  \|\cdot\|_{Y^0},\, \|\cdot\|_{X^0},$ which is an immediate
  consequence of (\ref{embeddingUpVp}) and (\ref{almostalthogonal}),
  and both spaces are Banach spaces under respective norms. Moreover
  using (\ref{almostalthogonal}) we can check the embeddings
  \begin{equation}\label{embeddingX0Y0}
    U^2_{h(D)}\hookrightarrow X^0\hookrightarrow Y^0 \hookrightarrow V^2_{h(D)}.
  \end{equation}
  Also, $X^0,Y^0$ behave well under time truncation and frequency
  truncation in the sense of (\ref{-99}), since $U^2,V^2$ spaces do.
\end{remark}

We can also consider the time restricted spaces $X^0([a,b)),$ where we
allow $b$ to be infinity.

\begin{definition}
  We let $X^0([a,b))$ be the space of functions $u:[a,b)\rightarrow
  L^2(\R^2)$ such that the zero extension $\tilde{u}$
  $$\tilde{u}(t)=
  \begin{cases}
    u(t) & \mbox{if } t\in [a,b) \\
    0 & \mbox{otherwise}.
  \end{cases}
  $$
  belongs to $X^0.$ We put norm
  $\|u\|_{X^0([a,b))}=\|\tilde{u}\|_{X^0}$ on the space $X^0([a,b)).$
\end{definition}

Since $\|\cdot\|_{L^\infty_tL^2_x} \lesssim \|\cdot\|_{X^0},$ we
conclude that under the identification $u\mapsto \tilde{u},$
$X^0([a,b))$ is a closed subspace of $X^0$ and therefore $X^0([a,b))$
is a Banach space with the norm $\|\cdot\|_{X^0([a,b))}.$ If we write
$X^0$ without specifying intervals, we always mean $X^0(\R).$

We also need the following duality property between $U^p$ and $V^{p'}$
spaces. This is a consequence of Theorem 2.10 and Remark 5 in
\cite{hadac2009well}, and Lemma 4.32 in \cite{kochtataruvisan2014dispersive}.
\begin{proposition}\label{duality}
  Suppose $u:\R\rightarrow L^2_x(\R^2)$ is absolutely continuous on compact intervals and $u'(t)=0$ on $(-\infty,0).$ Then we have
  \begin{equation}\label{p1}
    \|u\|_{U^2}=\sup_{v\in V^2_{rc}: \|v\|_{V^2}=1} \left|\int_\R \int_{\R^2} u'(t)v(t)dxdt\right|,
  \end{equation}
  by which we mean if the right hand side of (\ref{p1}) is finite then $u\in U^2$ and (\ref{p1}) holds.
\end{proposition}

Suppose $T\in [0,\infty].$ We let $I_T$ be the linear operator given
by
\begin{equation}\label{IToperator}
  I_T(f)(t,x)=\int_{0}^{t}1_{[0,T)}(s)e^{i(t-s)h(D)}f(s)ds.
\end{equation}
We claim that if $f\in L^1_{t,loc}L^2_x,$ then $e^{-ith(D)}I_T(f)$
satisfies the conditions in Proposition \ref{duality}. To show that $e^{-ith(D)}I_T(f)$ is
absolutely continuous on compact intervals we notice that
  $$\partial_t\left( e^{-ith(D)}I_T(f)\right)=1_{[0,T)}(t)e^{-ith(D)}f(t),$$
  and by the assumption on $f$ we have $1_{[0,T)}(t)e^{-ith(D)}f(t)\in
  L^1_{t,loc}L^2_x.$ Therefore by the fundamental theorem of calculus for
  Banach space valued functions we have that $e^{-ith(D)}I_T(f)$ is
  absolutely continuous on compact intervals. The previous computation also shows that
  $$\partial_t\left( e^{-ith(D)}I_T(f)\right)=0 \text{ on } (-\infty,0).$$

  As an application of Proposition \ref{duality} we prove the
  following estimate.
  \begin{proposition}\label{dualityX0}
    If $f\in L^1_{t,loc}L^2_x,$ and $T\in [0,\infty],$ then we have
    $$\|I_T(f)(t,x)\|_{X^0}\lesssim \sup_{\|v\|_{Y^0}\leq 1}
    \left|\int_{0}^{T}\int_{\R^2} f(t,x)\overline{v(t,x)}
      dxdt\right|.$$
  \end{proposition}
  \begin{proof}
    Fix $\sigma>0.$ Observe that since
    $$\|I_T(f)(t,x)\|_{X^0}=\left\| \left( \| P_k I_T(f) \|_{U^2_{h(D)}} \right)_{k\leq 0} \right\|_{l^2},$$
    there exists a sequence $(b_k)_{k\leq 0}$ in $\R$ such that
    $\|(b_k)\|_{l^2}=1$ and
    $$\|I_T(f)(t,x)\|_{X^0}\leq  \sum_{k} b_k \| P_k I_T(f) \|_{U^2_{h(D)}}+\sigma . $$
    Note that $P_kI_T(f)=I_T(P_kf)$ and $P_kf\in L^1_{t,loc}L^2_x.$ By
    Proposition \ref{dualityX0}, there exists $v_k\in V^2_{rc}$ with
    $\|v_k\|_{V^2}=1$ such that
    \begin{align*}
      \| I_T(P_kf) \|_{U^2_{h(D)}} & \leq \int_\R \int_{\R^2}1_{[0,T)}(t) e^{-ith(D)}  P_k f(t) \overline{v_k(t)}dxdt +2^k\sigma \\
      & = \int_\R \int_{\R^2}1_{[0,T)}(t) f(t) \overline{e^{ith(D)}
        P_k v_k(t)}dxdt +2^k\sigma.
    \end{align*}
    where we have used Plancherel's theorem to move the Fourier
    multiplier $e^{-ith(D)}P_k$ to act on $\overline{v_k}.$ Now we let
    $v=\sum_k b_k e^{ith(D)}P_k v_k.$ Then we have
    \begin{align*}
      \|I_T(f)(t,x)\|_{X^0} & \leq \sum_k b_k \| I_T(P_kf) \|_{U^2_{h(D)}}+\sigma \\
      & \leq \sum_k \int_\R \int_{\R^2}1_{[0,T)}(t) f(t) \overline{b_k e^{ith(D)}  P_k v_k(t)}dxdt +3\sigma \\
      & = \int_{0}^{T}\int_{\R^2} f(t)\overline{v(t)}dxdt +3\sigma.
    \end{align*}
    Also note that $P_Nv=P_{N}\left(\sum_{k=N-2}^{N+2}P_kv_k\right)$
    is right continuous since by assumption $v_k$ are, and by the
    triangle inequality and (\ref{-99}) we have
    \begin{align*}
      \|v\|_{Y^0}^2 & = \sum_N \left\| P_{N}  \left(\sum_{k=N-2}^{N+2}b_ke^{ith(D)}P_k v_k\right) \right\|_{V^2_{h(D)}}^2 \\
      & \lesssim \sum_N \left( \sum_{k=N-2}^{N+2} b_k^2\|e^{ith(D)}v_k\|_{V^2_{h(D)}}^2 \right) \\
      & \lesssim \sum_N \left( \sum_{k=N-2}^{N+2} b_k^2\|v_k\|_{V^2}^2 \right)  \\
      & \lesssim 1.
    \end{align*}
    Since we can always normalize $v$ (in a way that is independent of
    $\sigma$) such that $\|v\|_{Y^0}\leq 1$ and $\sigma$ is arbitrary,
    we obtain the desired result.
  \end{proof}
  \begin{remark}\label{ITL1L2remark}
    The proof also shows that if $f\in L^1_{t,loc}L^2_x,$ then for
    every $T\in [0,\infty],$
  $$\|I_T(f)\|_{X^0}\lesssim \|f\|_{L^1_t([0,T),L^2_x)},$$
  and the implicit constant is independent of $T.$ This is because of
  H\"{o}lder's inequality and the embedding $Y^0\subset L^\infty_t
  L^2_x.$
\end{remark}

\section{Localized Linear and Bilinear Estimates}{\label{estimatessection}}

In this section we consider linear and bilinear Strichartz estimates for our problem,
localized to dyadic anuli $\{ ||\xi| -1| \approx 2^k\}$.
We begin with the following localized Strichartz estimate.
\begin{proposition}\label{linearestimate}
  Suppose $k\leq M+1.$ Then we have
  \begin{equation}\label{-1}
      \left\|e^{ith(D)}P_{k}u_0\right\|_{L^4_{t,x}}\lesssim 2^{-\beta k/8}\left\|P_ku_0\right\|_{L^2_x}.
  \end{equation}
  As a consequence, for $k\leq M+1$ we have
  \begin{equation}\label{15}
    \left\|e^{ith(D)}P_{k\leq \cdot \leq M+1}u_0\right\|_{L^4_{t,x}}\lesssim 2^{-\beta k/8} \left\|P_{k\leq \cdot \leq M+1}u_0\right\|_{L^2_x}.
  \end{equation}
\end{proposition}

Before proving this proposition, we first establish the following localized dispersive estimate.

\begin{lemma}
  Let $\chi_k(D)$ be the Fourier multiplier $P_{k-2\leq \cdot \leq k+2}.$ Then we have
  \begin{equation}\label{dispersiveestimate}
    \left\|e^{ith(D)}\chi_k^2(D) u_0\right\|_{L^\infty_{x}}\lesssim \frac{2^{-\beta k/2}}{|t|}\|u_0\|_{L^1_x}.
  \end{equation}
\end{lemma}
\begin{proof}
  We have $\chi_k(D)=\chi_k^+(D)+\chi_k^-(D)$ where $\chi_k^{\pm}(D)$
  are Fourier multipliers with symbols $P_{k-2\leq \cdot \leq
    k+2}^{\pm}(\xi).$ We will only establish
  (\ref{dispersiveestimate}) for $\|e^{ith(D)}(\chi_k^+(D))^2
  u_0\|_{L^\infty_{x}}.$ The other part $\|e^{ith(D)}(\chi_k^-(D))^2
  u_0\|_{L^\infty_{x}}$ can be estimated in the same way and the
  triangle inequality will give the estimate
  (\ref{dispersiveestimate}), as $\chi_k^+ (D) \chi_k^-(D)=0.$ For
  simplicity of notation we will still write $\chi_k(D),$ omitting the
  $+$ sign.  We first observe that
  $$e^{ith(D)}\chi_k^2(D)u_0= K*u_0$$
  where
  $$K(t,x)=\widecheck{e^{ith(\xi)}\chi^2_k(\xi)}=\int_{\R^2}e^{ix\cdot \xi}e^{ith(\xi)} \chi_k^2(\xi)\frac{d\xi}{(2\pi)^2}.$$
  Therefore by Young's inequality if suffices to show that
  $$\|K(t,x)\|_{L^\infty_x}\lesssim \frac{2^{-\beta k/2}}{|t|}.$$
  Since
  $$\nabla_\xi (th(\xi)+x\cdot \xi)=t\gamma'(|\xi|)\frac{\xi}{|\xi|}+x,$$
  by our assumption (b) on $\gamma$ in Section \ref{introsection},
  there exists a small admissible constant $c>0$ such that
  \begin{equation}\label{50}
    \left| \nabla_\xi (th(\xi)+x\cdot \xi) \right| \geq c|x|
  \end{equation}
  for every $|t|\leq c|x|$ and $|\xi|\in [1-2^{M+3},1+2^{M+3}],$
  and
  \begin{equation}\label{1000}
    \left| \nabla_\xi (th(\xi)+x\cdot \xi) \right| \geq c|t|
  \end{equation}
  for every $|t|> c^{-1}|x|$ and $|\xi|\in [1-2^{M+3},1+2^{M+3}].$

  Write $x=(x_1,x_2),$ and $\xi=(\xi_1,\xi_2).$ We consider three cases.

  \medbreak
  \noindent
  {\bf Case 1. $0<|t|\leq c|x|.$ } We integrate by parts and obtain
  \begin{align*}
    K(t,x)& =
    \int_{\R^2}e^{ix\cdot \xi}e^{ith(\xi)} \chi_k^2(\xi)\frac{d\xi}{(2\pi)^2} \\ &
    =-i\int_{\R^2}\frac{\left(t\gamma'(|\xi|)\frac{\xi}{|\xi|}+x\right) \cdot \nabla_\xi e^{i(x\cdot \xi+th(\xi))}}{\left|t\gamma'(|\xi|)\frac{\xi}{|\xi|}+x\right|^2} \chi_k^2(\xi)\frac{d\xi}{(2\pi)^2} \\
    & =i\int_{\R^2}e^{i(x\cdot \xi+th(\xi))}\nabla_\xi \cdot \left(\frac{\left(t\gamma'(|\xi|)\frac{\xi}{|\xi|}+x\right)}{\left|t\gamma'(|\xi|)\frac{\xi}{|\xi|}+x\right|^2}  \chi_k^2(\xi) \right) \frac{d\xi}{(2\pi)^2}.
  \end{align*}
  For every $0<|t|\leq c|x|$ and $\xi \in \supp \chi_k,$ by (\ref{50}) we have for $j=1,2$
  $$\left| \partial_{\xi_j}
  \frac{\left(t\gamma'(|\xi|)\frac{\xi_j}{|\xi|}+x_j\right)}{\left|t\gamma'(|\xi|)\frac{\xi}{|\xi|}+x\right|^2}
  \right| \lesssim \frac{|t|+|x|}{|x|^2}+\frac{|t|^2+|x|^2}{|x|^3} \lesssim \frac{1}{|t|}.$$
  Therefore for $0<|t|\leq c|x|,$
  \begin{equation}\label{1001}
    \left| \int_{\R^2}e^{i(x\cdot \xi+th(\xi))}\left( \nabla_\xi \cdot \frac{\left(t\gamma'(|\xi|)\frac{\xi}{|\xi|}+x\right)}{\left|t\gamma'(|\xi|)\frac{\xi}{|\xi|}+x\right|^2} \right) \chi_k^2(\xi)  \frac{d\xi}{(2\pi)^2} \right| \lesssim \frac{1}{|t|}.
  \end{equation}
  By definition
  $$\chi_k(\xi)=\chi\left(2^{-k-2}(|\xi|-1))\right)-
  \chi\left(2^{-k+3}(|\xi|-1))\right).$$
  If we abuse our notation by writing $\chi_k(r)=\chi_k(|\xi|)$ for $|\xi|=r$ (note that $\chi_k$ is radial),
  then we have
  \begin{equation}\label{2000}
    \int \left| \chi_k'(r)\right|dr \lesssim 1,
  \end{equation}
  which implies that
  $$ \int \left| \nabla_\xi \chi_k^2(\xi) \right| d\xi \lesssim 1.$$
  Hence again by (\ref{50}) we conclude that for $0<|t|\leq c|x|,$
  \begin{equation}\label{1002}
    \left| \int_{\R^2}e^{i(x\cdot \xi+th(\xi))} \frac{\left(t\gamma'(|\xi|)\frac{\xi}{|\xi|}+x\right)}{\left|t\gamma'(|\xi|)\frac{\xi}{|\xi|}+x\right|^2} \cdot \left( \nabla_\xi \chi_k^2(\xi) \right)  \frac{d\xi}{(2\pi)^2} \right| \lesssim \frac{1}{|t|}.
  \end{equation}
  Combining (\ref{1001}) and (\ref{1002}) we obtain
  $$|K(t,x)|\lesssim \frac{1}{|t|}$$
  for every $0<|t|\leq c|x|.$

  \medbreak
  \noindent
  {\bf Case 2. $|t|> c^{-1}|x|.$ } We repeat the argument in Case 1
  with slight modifications and use (\ref{1000}) in the place of
  (\ref{50}) to obtain
  $$|K(t,x)|\lesssim \frac{1}{|t|}$$
  for every $|t|>c^{-1}|x|.$

  \medbreak
  \noindent
  {\bf Case 3. $c|x|<|t|\leq c^{-1}|x|.$ }
  We write the integral $K(t,x)$ in polar coordinate $\xi=re^{i\theta}$ as
  $$K(t,x)=\frac{1}{(2\pi)^2}\int_{0}^{2\pi} \int_{0}^{\infty} e^{i(r(x_1\cos \theta+x_2 \sin \theta)+t\gamma(r)) } \chi_k^2(r) r dr d\theta,$$
  where as before we have abused our notation by writing $\chi_k(r)$ for $\chi_k(re^{i\theta}).$
  Making a change of variable in $\theta,$ we have
  $$K(t,x)=\frac{1}{(2\pi)^2}\int_{0}^{2\pi} \int_{0}^{\infty} e^{i(r|x|\cos \theta+t\gamma(r)) } \chi_k^2(r) r dr d\theta.$$

  Let $\psi_1(\theta):\mathbb{S}^1=\R/2\pi \Z \rightarrow [0,1]$ be a smooth $2\pi-$periodic function on $\R$ such that
  $\psi_1|_{[-\pi,\pi]}=1$ on $[-\pi/4,\pi/4]$ and $\supp \psi_1|_{[-\pi,\pi]}\subset [-\pi/2,\pi/2],$ and let $\psi_2(\theta)$ be the function $\psi_1(\theta+\pi).$ Let $\psi_3=1-\psi_1-\psi_2.$

  We now write $K(t,r)$ as the sum $\left( I_1+I_2+I_3 \right)/(2\pi)^2$, where
  $$I_i=\int_{0}^{2\pi} \int_\R e^{i(r|x|\cos \theta+t\gamma(r)) } \chi_k^2(r)  \psi_i(\theta) r dr d\theta.$$
  We estimate the three integrals separately.

  \smallbreak
  \noindent
  {\bf Estimate of $|I_1|$ and $|I_2|.$ }
  We will only estimate $|I_1|$ since the same argument applies to $|I_2|.$
  Suppose $(t,x)=\lambda (t_0,x_0)$ with $\lambda=|(t,x)|.$ Note that $|t_0|\sim |x_0| \sim 1$ and $\lambda\sim |t|$ by our assumption $c|x|<|t|\leq c^{-1}|x|.$
  We write $I_1$ as
  $$\int_{-\pi}^{\pi} \int_\R e^{i\lambda(\tilde{r}|x_0|\cos \theta+t_0\gamma(\tilde{r})) } \chi_k^2(\tilde{r})  \psi_1(\theta) \tilde{r} d\tilde{r} d\theta.$$
  Due to the support properties of $\chi_k$ and $\psi_1,$ we can make the change of variable
  $$\begin{cases}
      r=\tilde{r} \\
      \theta=2\arcsin\left( \alpha \sqrt{(1/2r)} \right) \quad \left(\text{or } \alpha^2=r(1-\cos \theta) \right)
    \end{cases}$$
  in the integral and obtain
  \begin{multline*}
    |I_1|=\int_{-\sqrt{1+2^{M+3}}}^{\sqrt{1+2^{M+3}}} \int_{1-2^{M+3}}^{1+2^{M+3}} e^{-i\lambda \alpha^2}e^{i\lambda(r|x_0|+t_0\gamma(r))}\chi_k^2(r)r \\
    \psi_1\left(2\arcsin \left(\alpha\sqrt{1/(2r)}\right)\right)\frac{2}{\sqrt{2r-\alpha^2}}drd\alpha.
  \end{multline*}
  Therefore by the Van der Corput lemma (see e.g. \cite{stein1993harmonic}) applied to integration in $\alpha$ we have
  \begin{align*}
    |I_1| & \leq \frac{C}{\lambda^{1/2}}
  \int_{-\sqrt{1+2^{M+3}}}^{\sqrt{1+2^{M+3}}} I_{1,\alpha}d\alpha \lesssim \frac{C}{\lambda^{1/2}} \sup_{|\alpha|\leq \sqrt{1+2^{M+3}}}  I_{1,\alpha}
  \end{align*}
   where $C$ is a universal constant and
   $$I_{1,\alpha}=\left| \int_{1-2^{M+3}}^{1+2^{M+3}} e^{i\lambda(r|x_0|+t_0\gamma(r))}\chi_k^2(r)r  \partial_\alpha \left( \psi_1\left(2\arcsin \left(\alpha\sqrt{1/(2r)}\right)\right)\frac{1}{\sqrt{2r-\alpha^2}}\right)  dr \right|.$$
   To estimate $I_{1,\alpha},$ noting that $\left|\partial_r^2 (r|x_0|+t_0\gamma(r)) \right|\gtrsim 2^{\beta k}$ on $\supp \chi_k,$ we apply Van der Corput lemma again to the integration in $r$ and obtain
   $$I_{1,\alpha}\lesssim \frac{C2^{-\beta k/2}}{\lambda^{1/2}}
   \int_{1-2^{M+3}}^{1+2^{M+3}} \left| \partial_r \left( \chi_k^2(r)r  \partial_\alpha \left( \psi_1\left(2\arcsin \left(\alpha\sqrt{1/(2r)}\right)\right)\frac{1}{\sqrt{2r-\alpha^2}}\right)  \right) \right| dr .$$
   Recalling (\ref{2000}) and support properties of $\chi_k(r)$ and $\psi_1(\theta)$ we conclude
   $$\sup_{|\alpha|\leq \sqrt{1+2^{M+3}}}  I_{1,\alpha}\lesssim \frac{2^{-\beta k/2}}{\lambda^{1/2}},$$
   which implies
   $$|I_1|\lesssim \frac{2^{-\beta k/2}}{\lambda}\lesssim \frac{2^{-\beta k/2}}{|t|}.$$

  \smallbreak
  \noindent
  {\bf Estimate of $|I_3|.$ }
  Since in the support of $\psi_3(\theta)$ we have $|\sin \theta|\gtrsim 1,$ using integration by parts we obtain
  \begin{align*}
    |I_3| & = \left| \int_{0}^{2\pi} \int_\R
    \frac{\partial_\theta (e^{ir|x|\cos \theta})}{-ir |x|\sin \theta} e^{it\gamma(r)}\chi_k^2(r)  \psi_3(\theta) r dr d\theta \right| \\
     & =\left| \int_{0}^{2\pi} \int_\R
    e^{ir|x|\cos \theta} \partial_\theta \left(  e^{it\gamma(r)}
    \frac{\chi_k^2(r)  \psi_3(\theta) }{i |x|\sin \theta}
      \right) dr d\theta \right| \\
     & \leq \int_{0}^{2\pi} \int_\R \left|
    \partial_\theta \left(
    \frac{ \psi_3(\theta) }{ \sin \theta}
      \right) \right| \frac{\chi_k^2(r)}{|x|}dr d\theta \\
     & \lesssim \frac{1}{|x|}
  \end{align*}
  Therefore when $c|x|<t\leq c^{-1}|x|$ we conclude
  $$|I_3|\lesssim \frac{1}{|t|}.$$

  \smallbreak
  Hence we have
  $$\|K(t,x)\|_{L^\infty_x}\leq |I_1|+|I_2|+|I_3|\lesssim \frac{2^{-\beta k/2}}{|t|}$$
  and our proof is complete.
\end{proof}

\begin{proof}[Proof of Proposition \ref{linearestimate}]
  We first observe that by construction we have
  $$\chi_k(D)P_{k}(D)=P_{k}(D).$$
  Therefore it suffice to prove the estimate
  \begin{equation}\label{16}
    \left\|e^{ith(D)}\chi_k(D)u_0\right\|_{L^4_{t,x}}\lesssim 2^{-\beta k/8} \|u_0\|_{L^2_x}
  \end{equation}
  for every $u_0\in L^2_x.$
  By the standard $TT^*$ argument (see for example \cite{stein1993harmonic}, \cite{tao2006nonlinear}) applied to the operator $e^{ith(D)}\chi_k(D)$ we conclude that the estimate (\ref{16}) is equivalent to the following estimate
  $$\left\|\int e^{i(t-s)h(D)}\chi_k^2(D)F(s)ds \right\|_{L^{4}_{t,x}}\lesssim 2^{-\beta k/4}\|F\|_{L^{4/3}_{t,x}}.$$
  By duality it suffices to show
  \begin{equation}\label{30}
    \left| \int \int \left(\int e^{i(t-s)h(D)}\chi_k^2(D)F(s)ds \right) \overline{G(t,x)}dtdx \right| \lesssim 2^{-\beta k/4}\|F\|_{L^{4/3}_{t,x}}\|G\|_{L^{4/3}_{t,x}}.
  \end{equation}
  We can interpolate the trivial estimate
  $$\left\|e^{ith(D)}\chi_k^2(D) u_0\right\|_{L^2_x}\lesssim \|u_0\|_{L^2_x}$$
  with the dispersive estimate (\ref{dispersiveestimate}) to obtain
  \begin{equation}\label{18}
    \left\|e^{ith(D)}\chi_k^2(D) u_0\right\|_{L^{4}_x}\lesssim  \frac{2^{-\beta k/4}}{|t|^{1/2}} \|u_0\|_{L^{4/3}_x}.
  \end{equation}
  Therefore by H\"{o}lder's inequality and Plancherel' theorem we have
  \begin{align*}
    & \left| \int \int \left(\int e^{i(t-s)h(D)}\chi_k^2(D)F(s)ds \right) \overline{G(t,x)}dtdx \right|  \\ & \qquad \qquad \qquad \qquad \qquad \qquad \qquad
    =\left| \int \int \left(\int e^{i(t-s)h(D)}\chi_k^2(D)F(s)  \overline{G(t)}dx \right)dsdt \right| \\ & \qquad \qquad \qquad \qquad \qquad \qquad \qquad
    \lesssim \int \int \frac{2^{-\beta k /4}}{|t-s|^{1/2}} \| F(s)\|_{L^{4/3}_x}\|G(t)\|_{L^{4/3}_x}  dsdt.
  \end{align*}
  Since $\dfrac{1}{2}+\dfrac{3}{4}+\dfrac{3}{4}=1+1,$ by Hardy-Littlewood-Sobolev inequality we obtain the desired inequality (\ref{30}). (\ref{15}) follows from (\ref{-1}) and the triangle inequality.
\end{proof}


\begin{remark}
    Observe that the weighted restriction estimate
    \begin{equation}\label{g2}
      \left\|e^{ith(D)}\kappa^{1/8}P_{\leq M+3}u_0\right\|_{L^4_{t,x}}\lesssim \|P_{\leq M+3}u_0\|_{L^2_x},
    \end{equation}
    where $$\kappa(\xi)=\frac{\gamma'(|\xi|)\gamma''(|\xi|)}{|\xi|(1+(
      \gamma'(|\xi|) )^2 )^2}$$ is the Gaussian curvature of the
    surface $\tau=h(\xi)$, will imply the localized Strichartz
    estimates in Proposition~\ref{linearestimate}.  There are some
    prior results in this direction.  If our curve $\gamma$ is convex,
    then the weighted restriction theorem proved in
    \cite{shayya2007affine} immediately implies (\ref{-1}).  Other
    such weighted restriction theorems have been considered for some
    radial surfaces in $\R^3$ \cite{oberlin2004uniform}
    \cite{carbery2007restriction}.  See also
    \cite{kenig1991oscillatory} for some weighted restriction
    estimates in $\R^n.$ For more general hypersurfaces in $\R^n$, see
    \cite{sogge1985averages} \cite{cowling1990damping}. Conversely,
    for a radial surface in $\R^3$ given by a curve $\gamma$
    satisfying our assumptions, (\ref{g2}) is an easy consequence of
    (\ref{-1}) and the bilinear estimate (\ref{g1}) proved later. One
    may also prove (\ref{g2}) directly by modifying our proof for
    (\ref{-1}).

\end{remark}

By the atomic structure of $U^p$ spaces, and embedding (\ref{embeddingUpVp}), we obtain the following.
\begin{corollary}\label{linearestimateV2}
  For every $k\leq M+1,$ we have
  $$\|P_{k\leq \cdot \leq M+1} u\|_{L^4_{t,x}}\lesssim 2^{-k \beta/8} \|P_{k\leq \cdot \leq M+1} u\|_{U^4_{h(D)}},$$
  and
  $$\|P_{k\leq \cdot \leq M+1} u\|_{L^4_{t,x}}\lesssim 2^{-k \beta/8} \|P_{k\leq \cdot \leq M+1} u\|_{V^2_{h(D)}}.$$
\end{corollary}

\bigskip

The second part of the section is devoted to bilinear $L^2$ estimates,
which exploit transversality in the bilinea interaction of waves,
rather than dispersion. As a starting point for our analysis, we
utilize the following well-known bilinear estimate. The proof given
here is adapted from \cite{bourgain1998refinements}.




\begin{proposition}\label{bilinearestimateoriginal}
  Suppose $h(\xi),h'(\xi)$ are smooth functions on some open subset $\Omega$ of $\R^2$ and $\Omega'\Subset \Omega.$ Let $\Omega_1,\Omega_2$ be two open subset of $\Omega'.$ Suppose
  $$\theta :=\inf_{\xi\in \Omega_1, \eta\in \Omega_2} |\nabla h(\xi)-\nabla h'(\eta)|>0$$
  and
  $$l=\sup_{(\tau_0, \xi_0)\in \R\times \R^2} \mathrm{meas}_1\{(h(\xi),\xi): \xi \in \Omega_1, \xi_0-\xi\in \Omega_2, \, \tau_0-h(\xi)=h'(\xi_0-\xi)  \}<\infty.\footnote{The previous condition gurantees that the set being measured, which is the intersection of the graph of $h$ over $\Omega_1$ and an inverse translate of the graph of $h'$ over $\Omega_2,$ is a 1-dimensional submanifold in $\R^3.$ So here the 1-dimensional Haussdorff measure $\mathrm{meas}_1$ just means the length of the curve.}$$
  Then we have
  $$\|(e^{ith(D)}u)(e^{ith'(D)}v)\|_{L^2_{t,x}}\leq \theta^{-1/2}l^{1/2}\|u\|_{L^2_x}\|v\|_{L^2_x}$$
  for every $u,v\in L^2_x$ with $\supp \hat{u}\subset \Omega_1, \supp \hat{v}\subset \Omega_2.$
\end{proposition}
\begin{proof}
  By Plancherel's theorem and the Cauchy-Schwartz inequality (applied to the measure $\delta_0(\tau -h'(\xi-\xi')-h(\xi'))d\xi'$), we have
  \begin{align*}
    & \|(e^{ith(D)}u)(e^{ith'(D)}v)\|_{L^2_{t,x}}^2 =\|\int \delta_0(\tau -h'(\xi-\xi')-h(\xi'))\hat{u}(\xi-\xi')\hat{v}(\xi')d\xi' \|_{L^2_{\tau,\xi}}^2 \\
     & \leq  \left( \int\int\int |\hat{u}(\xi-\xi')|^2|\hat{v}(\xi')|^2 \delta_0(\tau -h'(\xi-\xi')-h(\xi')) d\xi' d\xi d\tau \right) \\
     & \qquad   \times \sup_{\tau, \xi} \left(\int 1_{\Omega_1}(\xi')1_{\Omega_2}(\xi-\xi') \delta_0(\tau -h'(\xi-\xi')-h(\xi')) d\xi'\right) \\
     & \leq \|u\|^2_{L^2_x} \|v\|_{L^2_x}^2 \sup_{\tau, \xi} \left(\int_{h'(\xi-\xi')+h(\xi')=\tau}1_{\Omega_1}(\xi')1_{\Omega_2}(\xi-\xi') \frac{d\mu(\xi')}{|\nabla h'(\xi-\xi')-\nabla h(\xi')|} \right).
  \end{align*}
  By our assumption we have
  $$
    \sup_{\tau, \xi} \left(\int_{h'(\xi-\xi')+h(\xi')=\tau}1_{\Omega_1}(\xi')1_{\Omega_2}(\xi-\xi') \frac{d\mu(\xi')}{|\nabla h'(\xi-\xi')-\nabla h(\xi')|} \right)
    \leq \theta^{-1}l.
  $$
  Therefore we obtain the claimed estimate.
\end{proof}

%

We now apply the above general bilinear estimate to our problem. For simplicity of exposition we assume without loss of generality that if $k_1\leq k_2-10,$ then $|\gamma''(r_1)|\leq |\gamma''(r_2)|/3$ whenever $|r_1-1|\leq 2^{k_1+2}$ and $2^{k_2-2}\leq |r_2-1|\leq 2^{M+3}.$ In general we choose a sufficiently large admissible integer in the place of $10.$

\begin{proposition}\label{bilinearestimatefreewave}
  Suppose $k_1\leq k_2-10$ and $k_2\leq M+1.$ Then we have
  \begin{equation}\label{g1}
    \left\|(e^{ith(D)}P_{\leq k_1}u) (e^{ith(D)}P_{k_2}v)\right\|_{L^2_tL^2_x}
    \lesssim 2^{-\beta k_2/4}2^{(k_1-k_2)/4} \left\|P_{\leq k_1}u\|_{L^2_x}\|P_{k_2}v\right\|_{L^2_x}.
  \end{equation}
  As a consequence we have
  $$
    \left\|(e^{ith(D)}P_{\leq k_1}u) (e^{ith(D)}P_{k_2\leq \cdot \leq M+1}v)\right\|_{L^2_tL^2_x}
    \lesssim 2^{-\beta k_2/4}2^{(k_1-k_2)/4} \left\|P_{\leq k_1}u\|_{L^2_x}\|P_{k_2\leq \cdot \leq M+1}v\right\|_{L^2_x}.
  $$
\end{proposition}
\begin{remark}
  Since taking complex conjugation preserves $L^2$ norms, we immediately deduce that the proposition still holds if both $e^{ith(D)}$ are replaced by $e^{-ith(D)}.$ The above proposition also holds if either $e^{ith(D)}$ on the left hand side is replaced by $e^{-ith(D)},$ in which case the following argument goes through with little modification.
\end{remark}

\begin{proof}
  The second estimate follows immediately from the first one by the triangle inequality. For (\ref{g1}) we will only estimate
  $ \left\|(e^{ith(D)}P_{\leq k_1}u) (e^{ith(D)}P_{k_2}^+v)\right\|_{L^2_tL^2_x}$
  since other combinations can be similarly estimated, and an application of the triangle inequality will give the desired result. As before we will omit $+$ in the notation. For simplicity of notation we will also abuse our notation by writing $u$ in place of $P_{\leq k_1}u$ and $v$ in place of $P_{k_2} v.$ $\hat{u}$ and $\hat{v}$ are supported on $\Omega_1$ and $\Omega_2$ respectively where
  \begin{align}\label{Omega}
    & \Omega_1=\{\xi\in \R^2:1-2^{k_1+2}< |\xi|< 1+2^{k_1+2}\} \\
    & \Omega_2=\{\xi\in \R^2:1+2^{k_2-2}< |\xi|< 1+2^{k_2+2}\}. \nonumber
  \end{align}

  We let $R_k$ be the Fourier multiplier given by the symbol $1_{k2^{k_1}\leq |\xi|-1< (k+1)2^{k_1}}(\xi).$ Then we can write
  $$v=\sum_{k=2^{k_2-k_1-3}}^{2^{k_2-k_1+3}} R_k v.$$
  We let $\Omega^k$ be the annulus $(k-1/3)2^{k_1}< |\xi|-1< (k+4/3)2^{k_1}.$
  By Plancherel's theorem if
  for every $\xi,\xi' \in \Omega_1$ and $\eta\in \Omega^{k},\eta' \in \Omega^{k'}$ we have $h(\xi)+h(\eta)\neq h(\xi')+h(\eta')$
  then
  $\langle (e^{ith(D)}u)(e^{ith(D)}R_kv), (e^{ith(D)}u)(e^{ith(D)}R_{k'}v) \rangle =0.$
  By our assumption (b) on $\gamma$ we conclude the following almost orthogonality relation
  \begin{equation}\label{c1}
    \|(e^{ith(D)}u)(e^{ith(D)}v)\|_{L^2}^2\lesssim \sum_k \|(e^{ith(D)}u)(e^{ith(D)}R_k v)\|_{L^2}^2.
  \end{equation}
  Now we fix $k\in [2^{k_2-k_1-3},2^{k_2-k_1+3}]\cap \Z.$ To estimate terms on the right hand side of (\ref{c1}), we consider the following two cases.

  \medbreak
  \noindent
  {\bf Case 1\, $k_1\geq (\beta+1)k_2 .\quad$}

  We let $L=(\beta+1)k_2.$
  For $L\leq m\leq -2$ we let $Q_{m,j}$ be the Fourier multiplier given by the symbol
  $$1_{j2^{m}2\pi\leq \xi/|\xi| <(j+1)2^{m}2\pi}(\xi),$$
  where we used the identification $\mathbb{S}^1\cong \R/2\pi \Z.$ Put another way, $Q_{m,j}(\xi)=1$ if and only if $\xi=|\xi|e^{i\theta}$ for some $\theta\in [j2^{m}2\pi,(j+1)2^{m}2\pi).$
  Then by the triangle inequality we have
  \begin{multline}\label{a2}
    \left\|(e^{ith(D)}u) (e^{ith(D)}R_k v)\right\|_{L^2} \\
    \leq
  \sum_{m=L}^{-2} \sum_{j=0}^{2^{-L}-1} \sum_{j':1<|j'-j|<\min \{6,2^{-m}-1 \}} \|(e^{ith(D)}Q_{m,j}u) (e^{ith(D)}Q_{m,j'}R_k v)\|_{L^2} \\
  +\sum_{j=0}^{2^{-L}-1} \sum_{j':|j'-j|\leq 1} \|(e^{ith(D)}Q_{m,j}u) (e^{ith(D)}Q_{m,j'}R_k v)\|_{L^2}
  \end{multline}
  We define
  $$U_{m,j}:=\left\{\xi\in \Omega_{1}: (j-1/3)2^{m}2\pi< \xi/|\xi| <(j+4/3)2^{m}2\pi \right\}$$
  $$V_{m,j'}:=\left\{\xi \in \Omega^{k}: (j'-1/3)2^{m}2\pi< \xi/|\xi| <(j'+4/3)2^{m}2\pi\right\}.$$
  Then $U_{m,j},V_{m,j'}$ contain the support of $Q_{m,j}(\xi)\hat{u}(\xi)$ and $Q_{m,j'}(\xi)R_k(\xi)\hat{v}(\xi)$ respectively.

  Note that when $\xi \in \Omega_{1}$ and $\xi' \in \Omega^k$ we have
  \begin{align*}
    |\nabla h(\xi)-\nabla h(\xi')| & =\left|\gamma'(|\xi|)\frac{\xi}{|\xi|}-\gamma'(|\xi'|)\frac{\xi'}{|\xi'|}\right| \\
     & = \left|\gamma'(|\xi|)\left( \frac{\xi}{|\xi|}-\frac{\xi'}{|\xi'|} \right)+ \left(\gamma'(|\xi|)-\gamma'(|\xi'|)\right)\frac{\xi'}{|\xi'|} \right| \\
     & \gtrsim \left|\frac{\xi}{|\xi|}-\frac{\xi'}{|\xi'|}\right| +2^{(\beta+1)k_2}
  \end{align*}
  because $\gamma'(|\xi|),\gamma'(|\xi'|) \sim 1$ and
  $$|\gamma'(|\xi|)-\gamma'(|\xi'|)|=\left|\int_{|\xi'|}^{|\xi|}\gamma''(r)dr\right|\sim 2^{(\beta+1)k_2}.$$
  So if we let $\theta$ be the parameter in Proposition \ref{bilinearestimateoriginal} for $U_{m,j}$ and $V_{m,j'}$
  then when $1<|j-j'|<6$ and $L\leq m\leq -4,$
  $$\theta \gtrsim \left|\frac{\xi}{|\xi|}-\frac{\xi'}{|\xi'|}\right| \gtrsim 2^{m}$$
  and when $|j'-j|\leq 1$ and $m=L,$
  $$\theta \gtrsim 2^{(\beta+1)k_2}= 2^L.$$

  We let $l$ be the parameter in Proposition \ref{bilinearestimateoriginal} for $U_{m,j}$ and $V_{m,j'}.$ Since $h$ has bounded $C^3-$norm on the annulus $\left||\xi|-1\right|\leq 2^{M+3},$ we have
  $$l\sim \sup_{\tau,\eta} \mathrm{meas}_1\, \{\xi\in U_{m,j}: \eta-\xi\in V_{m,j'}, h(\eta-\xi)+h(\xi)=\tau\}.$$
  We let $\gamma_{\tau,\eta}$ be the curve
  $$\{\xi\in U_{m,j}: \eta-\xi\in V_{m,j'}, h(\eta-\xi)+h(\xi)=\tau\}.$$
  Fix $\eta\in \R^2$ and $\tau\in \R.$
  To estimate the length of $\gamma_{\tau,\eta}$ we examine its defining function
  $$f(\xi):=h(\eta-\xi)+h(\xi).$$
  Then
  \begin{align}\label{a1}
     \nabla f(\xi) & =\gamma'(|\xi|)\frac{\xi}{|\xi|}- \gamma'(|\eta-\xi|)\frac{\eta-\xi}{|\eta-\xi|} \\
     & =(\gamma'(|\xi|)-\gamma'(|\eta-\xi|))
  \frac{\xi}{|\xi|}+\gamma'(|\eta-\xi|)\left(\frac{\xi}{|\xi|}-\frac{\eta-\xi}{|\eta-\xi|}\right).
  \nonumber
  \end{align}
  We write $\nabla f=\nabla_r f+\nabla_\theta f$ where
  $$\nabla_r f(\xi)=\left(\nabla f \cdot \frac{\xi}{|\xi|} \right) \xi, \qquad \nabla_\theta f=\nabla f-\nabla_r f$$
  are the radial and tangential parts of $\nabla f$ respectively. Note that when $\xi \in U_{m,j}$ and $\eta-\xi \in V_{m,j'},$
  $$-\gamma'(|\xi|)+\gamma'(|\eta-\xi|)\sim 2^{(\beta+1)k_2},\quad \gamma'(|\eta-\xi|)\sim 1$$
  So by (\ref{a1}) and our assumption $k_1\geq (\beta+1)k_2 $ we know that when $1<|j-j'|<6$
  $$
  \frac{|\nabla_r f(\xi)|}{|\nabla_\theta f(\xi)|}\begin{cases}
              \lesssim 1 & \mbox{if } k_1/2+(\beta+1)k_2/2<m\leq -4, \\
              \gtrsim \dfrac{2^{k_2(\beta
              +1)}}{2^{m}} & \mbox{if } L\leq m\leq k_1/2+(\beta+1)k_2/2.
            \end{cases}
  $$
  for every $\xi \in U_{m,j}$ satisfying $\eta -\xi \in V_{m,j'}.$
  So if we let $l$ be the parameter in Proposition \ref{bilinearestimateoriginal} for $U_{m,j}$ and $V_{m,j'},$ then\footnote{To be completely rigorous we need to show that the number of components of $\gamma_{\tau,\xi}$ is bounded by some admissible constant. See the proof of Lemma \ref{geometriclemma} below.}
  when $1<|j-j'|<6,$
  $$l \lesssim
  \begin{cases}
    2^{k_1} & \mbox{if } k_1/2+(\beta+1)k_2/2<m\leq -4, \\
    \dfrac{2^{m}}{2^{k_2(\beta+1)}}2^{m} & \mbox{if }L\leq m\leq k_1/2+(\beta+1)k_2/2
  \end{cases}
  $$
  and when $|j'-j|\leq 1$ and $m=L=(\beta+1)k_2,$
  $$l \lesssim 2^{L}= 2^{(\beta+1)k_2}.$$
  Applying Proposition \ref{bilinearestimateoriginal} we obtain
  when $1<|j-j'|<6,$
  \begin{multline*}
    \left\|(e^{ith(D)}Q_{m,j}u) (e^{ith(D)}Q_{m,j'} R_k v)\right\|_{L^2} \\ \lesssim
    \begin{cases}
    2^{k_1/2}2^{-m/2}\|Q_{m,j} u\|_{L^2}\|Q_{m,j'}R_k v\|_{L^2}& \mbox{if } k_1/2+(\beta+1)k_2/2<m\leq -4, \\
    2^{m-k_2(\beta+1)/2}2^{-m/2}\|Q_{m,j} u\|_{L^2}\|Q_{m,j'}R_k v\|_{L^2} & \mbox{if }L\leq m\leq k_1/2+(\beta+1)k_2/2
    \end{cases}
    \\ =
    \begin{cases}
    2^{k_1/2}2^{-m/2}\|Q_{m,j}u\|_{L^2}\|Q_{m,j'}R_k v\|_{L^2} & \mbox{if } k_1/2+(\beta+1)k_2/2<m\leq -4, \\
     2^{m/2}2^{-k_2(\beta+1)/2}\|Q_{m,j}u\|_{L^2}\|Q_{m,j'}R_k v\|_{L^2} & \mbox{if }L\leq m\leq k_1/2+(\beta+1)k_2/2
  \end{cases}
  \end{multline*}
  and when $|j'-j|\leq 1$ and $m=L,$
  \begin{align*}
    \left\|(e^{ith(D)}Q_{m,j}u) (e^{ith(D)}Q_{m,j'}R_k v)\right\|_{L^2} & \lesssim
  2^{L/2}2^{-L/2}\|Q_{m,j} u\|_{L^2}\|Q_{m,j'}R_k v\|_{L^2} \\
     & \lesssim \|Q_{m,j}u\|_{L^2}\|Q_{m,j'}R_k v\|_{L^2} \\
     & \lesssim 2^{-\beta k_2/4} 2^{(k_1-k_2)/4}\|Q_{m,j} u\|_{L^2}\|Q_{m,j'}R_k v\|_{L^2}
  \end{align*}
  where we have used our assumption $k_1\geq (\beta+1)k_2 .$

  When $1<|j-j'|<\min \{6,2^{-m}-1 \}$ and $-3\leq m\leq -2$ we can easily see that $\theta \gtrsim 1.$ By Lemma \ref{geometriclemma} below we have $l\lesssim 2^{(k_1-k_2)/2}.$ So Proposition \ref{bilinearestimateoriginal} gives
  \begin{equation}\label{d1}
    \left\|(e^{ith(D)}Q_{m,j}u) (e^{ith(D)}Q_{m,j'}R_k v)\right\|_{L^2}\lesssim 2^{(k_1-k_2)/4}\|Q_{m,j} u\|_{L^2}\|Q_{m,j'}R_k v\|_{L^2}.
  \end{equation}

  Therefore combining the above estimates, (\ref{a2}), and the Cauchy-Schwartz inequality yields
  \begin{align*}
   & \left\|(e^{ith(D)}u) (e^{ith(D)} R_k v)\right\|_{L^2} \\
     & \lesssim \left(\sum_{k_1/2+(\beta+1)k_2/2<m\leq -4} 2^{k_1/2}2^{-m/2} +\sum_{L\leq m\leq k_1/2+(\beta+1)k_2/2}2^{m/2}2^{-k_2(\beta+1)/2} \right.\\
     & \qquad  \left.+ 2^{(k_1-k_2)/4}+
     2^{-\beta k_2/4} 2^{(k_1-k_2)/4} \right) \sup_m \left( \sum_j \|Q_{m,j}u\|_{L^2}^2 \right)^{1/2}\left( \sum_{j'} \|Q_{m,j'}R_k v\|_{L^2}^2\right)^{1/2} \\
     & \lesssim \left(2^{k_1/2}2^{-(\beta+1)k_2/4-k_1/4}+2^{(\beta+1)k_2/4}2^{k_1/4}2^{-k_2(\beta+1)/2}+
     2^{-\beta k_2/4} 2^{(k_1-k_2)/4}\right)\|u\|_{L^2}\|R_k v\|_{L^2} \\
     & \lesssim 2^{-\beta k_2/4}2^{(k_1-k_2)/4} \|u\|_{L^2}\|R_k v\|_{L^2}.
  \end{align*}

  \medbreak
  \noindent
  {\bf Case 2.2\, $k_1<(\beta+1)k_2 . \quad$}
  We let $L'=k_1/2+(\beta+1)k_2/2$ (without loss of generality we may assume that $k_2$ is even).
  By the triangle inequality we have
  \begin{multline}\label{a3}
    \left\|(e^{ith(D)}u) (e^{ith(D)} R_k v)\right\|_{L^2} \\
    \leq
  \sum_{m=L'}^{-2} \sum_{j=0}^{2^{-m}-1} \sum_{j':1<|j'-j|<\min \{6,2^{-m}-1 \}} \|(e^{ith(D)}Q_{m,j}u) (e^{ith(D)}Q_{m,j'} R_k v)\|_{L^2} \\
  +\sum_{j=0}^{2^{-L'}-1} \sum_{j':|j'-j|\leq 1} \|(e^{ith(D)}Q_{m,j}u) (e^{ith(D)}Q_{m,j'} R_k v)\|_{L^2}.
  \end{multline}

  Arguing as before and noting the assumption $k_1<(\beta+1)k_2 ,$ we have
  when $1< |j-j'|<6$ and $L'\leq m\leq -4,$
  \begin{align*}
    & \left\|(e^{ith(D)}Q_{m,j} u) (e^{ith(D)}Q_{m,j'}  R_k v)\right\|_{L^2} \\ & \qquad \lesssim
    \left(\left(2^{k_1} \right)^{1/2} (2^{m})^{-1/2} + \left( \frac{2^{(\beta+1)k_2}}{2^{m}}2^{k_1}\right)^{1/2}2^{-(\beta+1)k_2/2}\right) \|Q_{m,j}  u\|_{L^2}\|Q_{m,j'} R_k v\|_{L^2}
    \\ & \qquad \lesssim 2^{k_1/2}2^{-m/2}
    \|Q_{m,j}  u\|_{L^2}\|Q_{m,j'} R_k v\|_{L^2},
  \end{align*}
  and when $|j'-j|\leq 1$ and $m=L'=k_1/2+(\beta+1)k_2/2,$
  \begin{align*}
    & \left\|(e^{ith(D)}Q_{m,j} u) (e^{ith(D)}Q_{m,j'}  R_k v)\right\|_{L^2} \\  & \qquad \qquad \qquad \qquad \qquad \lesssim
  \left(2^{(\beta+1)k_2/2}2^{k_1/2}\right)^{1/2}\left(2^{(\beta+1)k_2}\right)^{-1/2}
  \|Q_{m,j}  u\|_{L^2}\|Q_{m,j'} R_k v\|_{L^2} \\ & \qquad \qquad \qquad \qquad \qquad
      \lesssim 2^{-\beta k_2/4} 2^{(k_1-k_2)/4}\|Q_{m,j}  u\|_{L^2}\|Q_{m,j'} R_k v\|_{L^2}.
  \end{align*}
  Therefore combining the above estimates, (\ref{d1}) and the Cauchy-Schwarz inequality, we deduce from (\ref{a3}) that
  \begin{align*}
   & \left\|(e^{ith(D)} u) (e^{ith(D)}  R_k v)\right\|_{L^2} \\
     & \qquad \qquad \qquad \lesssim \left(2^{(k_1-k_2)/4}+2^{-\beta k_2/4}2^{(k_1-k_2)/4}+\sum_{ L'\leq m\leq -4} 2^{k_1/2}2^{-m/2}  \right)\| u\|_{L^2}\| R_k v\|_{L^2} \\
     & \qquad \qquad \qquad \lesssim \left(2^{k_2/2}2^{-(\beta+1)k_2/4-k_1/4}+
     2^{-\beta k_2/4} 2^{(k_1-k_2)/4}\right)\| u\|_{L^2}\| R_k v\|_{L^2} \\
     & \qquad \qquad \qquad \lesssim 2^{-\beta k_2/4}2^{(k_1-k_2)/4} \| u\|_{L^2}\| R_k v\|_{L^2}.
  \end{align*}

  \medbreak
  Combining Case 1 and Case 2 we therefore conclude that
  $$\|(e^{ith(D)}u) (e^{ith(D)}R_k v)\|_{L^2}\lesssim 2^{-\beta k_2/4}2^{(k_1-k_2)/4} \|u\|_{L^2}\|R_k v\|_{L^2}$$
  for every $k.$
  Now the almost orthogonality relation (\ref{c1}) implies
  $$\|(e^{ith(D)}u) (e^{ith(D)}v)\|_{L^2}\lesssim 2^{-\beta k_2/4}2^{(k_1-k_2)/4} \|u\|_{L^2}\|v\|_{L^2},$$
  as desired.
\end{proof}

\begin{lemma}\label{geometriclemma}
    Suppose $k_1\leq k_2-10$ and $k_2\leq M+1.$
    Let $\Omega_1,\Omega_2$ be the annuli defined by (\ref{Omega}). For every $\tau_0\in \R, \xi \in \R^2$ define
    $$\gamma_{\tau_0,\xi_0}:=\{\xi: \xi \in \Omega_1, \xi_0-\xi\in \Omega_2, \, \tau_0-h(\xi)=h(\xi_0-\xi) \}.$$
    Then
    $$\Theta:=\{\theta\in [0,2\pi]: \text{ there exists } r \text{ such that } re^{i\theta}\in \gamma_{\tau_0,\xi_0}\}$$
    is a disjoint union of at most $c$ many intervals, each of which has size less than $C2^{(k_1-k_2)/2}$ for some admissible constants $c,C.$
  \end{lemma}
  \begin{proof}
    We fix $\xi_0,\tau_0.$ Without loss of generality we may suppose that in polar coordinates $\xi_0=r_0e^{i\theta_0}=r_0.$
    Let $\Omega$ be the set
    $$\Omega:=\{\xi: \xi \in \Omega_1, \xi_0-\xi\in \Omega_2\}.$$
    and let $\Omega'$ be the set
    $$\Omega':=\{re^{i\theta}\in \Omega: \theta \neq 0,\pi\}.$$
    So using the above notation we have
    $$\gamma_{\tau_0,\xi_0}=\{\xi\in \Omega:\tau_0-h(\xi)=h(\xi_0-\xi)\}.$$
    Writing down the inequalities defining annuli $\Omega_1$ and $\xi_0-\Omega_2$ we see that, depending on $r_0,$ $\Omega'$ can have at most two components, which we will denote by $U'$ and $U''. $ They satisfy
    $$U'= \{re^{i\theta}\in \Omega: 0< \theta< \pi\},\quad U''=\{re^{i\theta}\in \Omega: -\pi< \theta <0 \}.$$
    In the case when there is only one component, we put the other one to be the empty set. If $\Omega'$ has no component, we put both to be the empty set.

    We first observe that
   $U'$ satisfies
   \begin{equation}\label{-10}
     U'_{\theta}:=\{r: re^{i\theta}\in U'\},\quad U'_{r}:=\{\theta: re^{i\theta}\in U'\}
   \end{equation}
  are intervals, for every $\theta,r.$ $U''$ has the same properties (\ref{-10}) with $U'$ replaced by $U''.$
  From now on we will only consider the part of $\gamma_{\tau_0,\xi_0}$ that lies in $U',$ and the same argument will apply to the part lying in $U''.$

  Let $f(r,\theta)$ be the defining function of $\gamma_{\tau_0,\xi_0}$
  $$f(r,\theta):=\gamma(|re^{i\theta}-r_0|)+\gamma(r).$$
  Noting that
  $$|re^{i\theta}-r_0|=\sqrt{r^2+r_0^2-2rr_0\cos \theta},$$
  we differentiate $f$ with respect to $r,\theta$ respectively and obtain
  \begin{equation}\label{-91}
    \partial_r f(r,\theta)= \gamma'(|re^{i\theta}-r_0|)\frac{r- r_0\cos \theta}{|re^{i\theta}-r_0|}+\gamma'(r),
  \end{equation}
  \begin{equation}\label{-90}
    \partial_\theta f(r,\theta)= \gamma'(|re^{i\theta}-r_0|)\frac{rr_0\sin \theta}{|re^{i\theta}-r_0|}.
  \end{equation}
  Recall that under the assumption $\xi=re^{i\theta}\in \Omega,$ we always have
  \begin{equation}\label{100}
    \gamma'(r),\, \gamma'(|re^{i\theta}-r_0|) \sim 1.
  \end{equation}
  Also note that in order for $\gamma_{\tau_0,\xi_0}$ to be nonempty, we must have
  \begin{equation}\label{101}
    2^{k_2}\lesssim r_0\lesssim 1.
  \end{equation}

  When $re^{i\theta}\in U',$ we have
  $$\partial_\theta f(r,\theta)>0.$$
  Therefore due to property (\ref{-10}), for every $r\in [1-2^{k_1+2},1+2^{k_1+2}],$ there exists at most one $\theta \in (0,\pi)$ such that $re^{i\theta}\in U'$ and $f(r,\theta)=\tau_0.$ So we can parametrize
  $$\gamma_1:= \gamma_{\tau_0,\xi_0} \cap U'$$
  by $\theta=\theta(r).$ By the implicit function theorem we know that $\theta(r)$ is a smooth function defined on a union of disjoint intervals $I:=\bigcup_i I_i$ in $(1-2^{k_1+2},1+2^{k_1+2}),$ and
  \begin{equation}\label{thetaprime}
    \theta'(r)=
  \frac{\gamma'(|re^{i\theta}-r_0|)(r-r_0\cos\theta)+\gamma'(r)|re^{i\theta}-r_0|}{\gamma'(|re^{i\theta}-r_0|)rr_0\sin \theta}.
  \end{equation}
  We write $I_i=(a_i,b_i)$ where $1-2^{k_1+2}\leq a_i< b_i\leq 1+2^{k_1+2}.$

  We fix $i$ in the index set $J$ of $I_i$ and $r_i \in I_i$ By (\ref{thetaprime}) we have
  $$(\cos \theta)'(r)=-\frac{\gamma'(|re^{i\theta}-r_0|)(r-r_0\cos\theta)+\gamma'(r)|re^{i\theta}-r_0|}{rr_0\gamma'(|re^{i\theta}-r_0|)},$$
  which implies for $r\in (a_i,b_i)$
  $$\cos \theta (r)=\cos \theta (r_i)-\int_{r_i}^{r}
  \frac{\gamma'(|re^{i\theta}-r_0|)(r-r_0\cos\theta)+\gamma'(r)|re^{i\theta}-r_0|}{rr_0\gamma'(|re^{i\theta}-r_0|)}dr.$$
  Since for every $r\in (a_i,b_i)$
  $$\left|
  \frac{\gamma'(|re^{i\theta}|)(r-r_0\cos\theta)+\gamma'(r)|re^{i\theta}-r_0|}{rr_0\gamma'(|re^{i\theta}-r_0|)}
  \right| \lesssim 2^{-k_2}$$
  and $b_i-a_i\lesssim 2^{k_1},$
  we therefore conclude that for every $r\in (a_i,b_i)$
  $$|\theta(r)-\theta(r_i)|\lesssim 2^{(k_1-k_2)/2}.$$

  To finish the proof we are left to show that the cardinality of the index set $J$ is bounded by some admissible constant $c.$
  Applying implicit function theorem to $f(r,\theta)$ on $\Omega'+B(0,2^{k_1}),$ the $2^{k_1}$ neighborhood of $\Omega',$ we see that
  for every $j\in J,$  $a_je^{i\theta(a_j)},b_je^{i\theta(b_j)}$ are contained in $\partial \Omega',$ and for distinct $j,k\in J,$ $\{a_je^{i\theta(a_j)},b_je^{i\theta(b_j)}\}\neq \{a_ke^{i\theta(a_k)},b_ke^{i\theta(b_k)}\}.$
  So it suffices to show that
  $$K:=\{re^{i\theta}\in \partial \Omega': \tau_0-\gamma(r)= \gamma \left(\left|r_0-re^{i\theta}\right|\right)\}$$
  has cardinality no more than $c$ for some admissible constant $c.$
  Note that $\partial \Omega'$ is a subset of $\partial \Omega_1 \cup
  (r_0+\partial \Omega_2)\cup \{re^{i\theta}\in \overline{\Omega_1}: \theta=0,\pi \}.$
  Previous argument has already shown that
  $$|K\cap \partial \Omega_1|\leq 2.$$
  Similarly we have\footnote{One way to see this is to repeat the previous argument in polar coordinates centered at $\xi_0.$}
  $$|K\cap(r_0+ \partial \Omega_2)|\leq 2.$$
  Setting $\theta =0$ in (\ref{-91}) yields
  $$\partial_r f(r,0)=\gamma'(|r-r_0|)\frac{r-r_0}{|r-r_0|}+\gamma'(r).$$
  Therefore for $re^{i0}\in \partial \Omega'$
  $$\partial_{rr} f(r,0)=\gamma''(r)+\gamma''(|r_0-r|)$$
  is either always positive or always negative, as $|\gamma''(|r_0-r|)|\geq |\gamma''(r)|/3.$
  Noting that $\{r: re^{i0}\in \partial \Omega'\}$ is an interval, we therefore conclude that
  $$|K\cap \{re^{i\theta}\in \overline{\Omega_1}:\theta=0\}|\leq 2.$$
  Similarly we have
  $$|K\cap \{re^{i\theta}\in \overline{\Omega_1}:\theta=\pi\}|\leq 2.$$
  Therefore $|K|\leq 8$ and the proof is complete.

\end{proof}

Using the extension property of $U^p$ spaces (Proposition 2.16 in \cite{hadac2009well}), we obtain
\begin{corollary}\label{bilinearestimateU2}
  Suppose $k_1\leq k_2-10$ and $k_2\leq M+1.$ Then we have
  \begin{equation}\label{2}
    \left\|P_{\leq k_1} u P_{k_2\leq \cdot \leq M+1}v\right\|_{L^2_{t,x}}\lesssim 2^{-\beta k_2/4} 2^{(k_1-k_2)/4}\left\|P_{\leq k_1}u\right\|_{U^2_{h(D)}}\left\|P_{k_2\leq \cdot \leq M+1}v\right\|_{U^2_{h(D)}}.
  \end{equation}
\end{corollary}

The right hand side of (\ref{2}) is evaluated in $U^2-$based spaces, which, as we shall see later, is insufficient for our fixed point argument. However,  we can use the following interpolation result from \cite{hadac2009well} to transfer from $U^p-$based estimates to $V^p-$based estimates, provided we have an additional $U^q-$based estimate for some $q>p.$ Another way of transferring from $U^p$ to $V^p$ can be found in  \cite{herrcandy2016transference}, which is not easy to implement here.
\begin{proposition}\label{U2toV2}
  Suppose $1< p< q.$ Let $E$ be a Banach space and $T:U^q\rightarrow E$ be a linear bounded operator with operator norm $C_q.$ Suppose in addition there exists $C_p\in (0,C_q]$ such that the estimate $\|Tu\|_{E}\leq C_p\|u\|_{U^p}$ holds for all $u\in U^p.$ Then we have for every $u\in V^p_{rc}$
  $$\|Tu\|_E\leq C_{p,q}C_p\left(\log \frac{C_q}{C_p}+1\right)\|u\|_{V^p},$$
  where $C_{p,q}$ is a constant depending only on $p,q.$
\end{proposition}

In our case we already have a $U^2-$based estimate in Corollary \ref{bilinearestimateU2}. We can also obtain a $U^4-$based estimate from Proposition \ref{linearestimate}. Indeed, under the same condition as in Corollary \ref{bilinearestimateU2}, by H\"{o}lder's inequality we have
\begin{align*}
  \|P_{k_1} u P_{k_2\leq \cdot \leq M+1}v\|_{L^2_{t,x}} & \lesssim \|P_{ k_1} u \|_{L^4_{t,x}} \| P_{k_2\leq \cdot \leq M+1}v\|_{L^4_{t,x}} \\
   & \lesssim  2^{-\beta (k_1+k_2)/8}\|P_{k_1}u\|_{U^4_{h(D)}}\|P_{k_2\leq \cdot \leq M+1}v\|_{U^4_{h(D)}}.
\end{align*}
Therefore applying Proposition \ref{U2toV2} twice\footnote{In fact we apply Proposition \ref{U2toV2} twice to obtain the estimate for enlarged projections: $$\|P_{k_1-2\leq \cdot \leq k_1+2} u P_{k_2-2\leq \cdot \leq M+3}v\|_{L^2_{t,x}}\lesssim 2^{-\beta k_2/4} 2^{(k_1-k_2)/4} |k_2-k_1|^2 \|u\|_{V^2_{h(D)}}\|v\|_{V^2_{h(D)}}.$$ Then we substitute $u$ and $v$ by $P_{k_1}u$ and $P_{k_2\leq \cdot \leq M+1}v$ respectively.}, and noticing the assumption that $k_1\leq k_2-10,$ we obtain
$$\|P_{k_1} u P_{k_2\leq \cdot \leq M+1}v\|_{L^2_{t,x}}\lesssim 2^{-\beta k_2/4} 2^{(k_1-k_2)/4} |k_2-k_1|^2 \|P_{k_1}u\|_{V^2_{h(D)}}\|P_{k_2\leq \cdot \leq M+1}v\|_{V^2_{h(D)}}$$
if $P_{k_1}u, P_{k_2\leq \cdot \leq M+1} v\in V^2_{h(D),rc}.$
Summing the above inequality over $k_1$ and using the triangle inequality we obtain the following corollary:
\begin{corollary}\label{bilinearestimate3V2}
  Suppose $k_1\leq k_2-10$ and $k_2\leq M+1.$ If $P_{\leq k_1}u,\, P_{k_2\leq \cdot \leq M+1}v\in V^2_{rc, h(D)},$ then
  \begin{equation}\label{-95}
    \|P_{\leq k_1} u P_{k_2\leq \cdot \leq M+1}v\|_{L^2_{t,x}}\lesssim 2^{-\beta k_2/4}\|P_{\leq k_1}u\|_{V^2_{h(D)}}\|P_{k_2\leq \cdot \leq M+1}v\|_{V^2_{h(D)}}
  \end{equation}
\end{corollary}


The localized linear and bilinear estimates discussed above allow us to obtain the following crucial estimate on the nonlinear term, whose proof will occupy the next section. We recall that $I_T$ is the operator given by
$$I_T(f)=\int_{0}^{t}1_{[0,T)}(s) e^{i(t-s)h(D)}f(s)ds.$$
\begin{proposition}\label{nonlinearestimate}
  For $T\in [0,\infty]$ we have
  \begin{equation}\label{nonlinearestimateeqn}
    \left\|I_T\left(A(D)\left(P_{\leq M}u_1  \overline{P_{\leq M}u_2 } P_{\leq M}u_3 \right)\right)\right\|_{X^0}\lesssim \|u_1\|_{X^0}\|u_2\|_{X^0}\|u_3\|_{X^0}.
  \end{equation}
\end{proposition}

\section{Proof of Proposition \ref{nonlinearestimate}}
If $k_1,k_2,k_3$ are three nonpositive integers, then we let $\tilde{k}_1,\tilde{k}_2,\tilde{k}_3$ be the increasing rearrangement of $k_1,k_2,k_3,$ that is, $\tilde{k}_1\geq \tilde{k}_2\geq \tilde{k}_3.$
Observe that when $u_1,u_2,u_3\in X^0,$ we have
$$A(D)\left(P_{\leq M}u_1(s) \overline{P_{\leq M}u_2(s)} P_{\leq M}u_3(s)\right)\in L^1_{t,loc}L^2_x$$
because of H\"{o}lder's inequality, the embedding $X^0\subset L^\infty_tL^2_x,$ and Bernstein's inequality.
Therefore by Proposition \ref{dualityX0} we have
\begin{multline*}
  \left\|I_T\left(A(D)\left(P_{\leq M}u_1 \overline{P_{\leq M}u_2} P_{\leq M}u_3\right)\right)\right\|_{X^0} \\
  =\sup_{\|v\|_{Y^0}\leq 1} \int_{0}^{T}\int_{\R^2} \left| \left(A(D)\left(P_{\leq M}u_1 \overline{P_{\leq M}u_2} P_{\leq M}u_3\right)\right)\overline{v} dxdt \right|.
\end{multline*}
We fix $v\in Y^0$ with $\|v\|_{Y^0}\leq 1.$
By Plancherel's theorem we can move the Fourier multiplier $A(D)$ to act on $v.$ Therefore we have
\begin{align}\label{5}
   & \left| \int_{0}^{T}\int_{\R^2} \left(A(D)\left(P_{\leq M}u_1 \overline{P_{\leq M}u_2} P_{\leq M}u_3\right)\right)\overline{v}dxdt \right| \\
   & \qquad \qquad \qquad \qquad \qquad \qquad =\left| \int_{0}^{T}\int_{\R^2} \left(P_{\leq M}u_1 \overline{P_{\leq M}u_2} P_{\leq M}u_3\right)\overline{A(D)v}dxdt \right|  \nonumber \\
   & \qquad \qquad \qquad \qquad \qquad \qquad \leq \sum_{k\leq M} \left|\int_{0}^{T}\int_{\R^2} \left(P_{\leq M}u_1 \overline{P_{\leq M}u_2} P_{\leq M}u_3\right)\overline{A(D)P_kv}dxdt \right|. \nonumber
\end{align}
Above we have used the support property of $A(\xi)$ to restrict the sum over $k\in \Z$ to $k\leq M.$

To shorten notation, we let
$$f_{k_1,k_2,k_3,k}:=P_{k_1}u_1 \overline{P_{k_2}u_2} P_{k_3}u_3\overline{A(D)P_kv}.$$
We similarly define  $f_{k_1\leq M;k_2\leq M;k_3\leq M;k}$ and others in the obvious way.

We can split the last sum in (\ref{5}) into four parts $I_1,I_2,I_3,I_4,$
where
\begin{align*}
   & I_1=\sum_{k\leq M} \left|\int_{0}^{T}\int_{\R^2} \sum_{k-10 \leq \tilde{k_1},\tilde{k_2},\tilde{k_3}\leq M} f_{k_1,k_2,k_3,k}dxdt \right| \\
   & I_2=\sum_{k\leq M}\,\left|\int_{0}^{T}\int_{\R^2} \sum_{k-10 \leq \tilde{k}_1,\tilde{k}_2 \leq M; \tilde{k}_3< k-10} f_{k_1,k_2,k_3,k}dxdt \right| \\
   & I_3=\sum_{k\leq M}\,\left|\int_{0}^{T}\int_{\R^2} \sum_{k-10\leq \tilde{k}_1\leq M; \tilde{k}_2, \tilde{k}_3< k-10} f_{k_1,k_2,k_3,k}dxdt \right| \\
   & I_4=\sum_{k\leq M}\, \left|\int_{0}^{T}\int_{\R^2} \sum_{\tilde{k}_1,\tilde{k}_2,\tilde{k}_3 < k-10} f_{k_1,k_2,k_3,k}dxdt \right|.
\end{align*}

We estimate the four sums separately. Our general strategy will be to use Corollary \ref{linearestimateV2} and Corollary \ref{bilinearestimate3V2} to establish the estimates
$$I_i\leq \|u_1\|_{X^0}\|u_2\|_{X^0}\|u_3\|_{X^0}$$
for $i=1,2,3,4,$ and therefore conclude Proposition \ref{nonlinearestimate}.

\medbreak

\noindent {\bf Estimate of $I_1:$}

First we consider the near diagonal sum, that is, the sum over $k_1,k_2,k_3\in [k-10,k+10).$ In this case
we use H\"{o}lder's inequality and Corollary \ref{linearestimateV2} four times to conclude that
\begin{align*}
   \left|\int_{0}^{T} \int_{\R^2}f_{k_1,k_2,k_3,k}dxdt\right| &
  = \left|\int_{0}^{T} \int_{\R^2}P_{k_1}u_1 \overline{P_{k_2}u_2} P_{k_3}u_3 \overline{P_kA(D) v}dxdt\right|  \\
   & \lesssim \|P_{k_1}u_1\|_{L^4_{t,x}}\|P_{k_2}u_2\|_{L^4_{t,x}}
   \|P_{k_3}u_3\|_{L^4_{t,x}}\|P_{k}A(D)v\|_{L^4_{t,x}} \\
   & \lesssim 2^{-4k\beta/8} 2^{\beta k/2} \|P_{k_1}u_1\|_{V^2_{h(D)}}
   \|P_{k_2}u_2\|_{V^2_{h(D)}}\|P_{k_3}u_3\|_{V^2_{h(D)}}
   \|P_{k}v\|_{V^2_{h(D)}}\\
   & \lesssim  \|P_{k_1}u_1\|_{V^2_{h(D)}}
   \|P_{k_2}u_2\|_{V^2_{h(D)}}\|P_{k_3}u_3\|_{V^2_{h(D)}}
   \|P_{k}v\|_{V^2_{h(D)}}.
\end{align*}
We then need to estimate the sum over the near diagonal region $k_1,k_2,k_3\in [k-10,k+10).$ Using the Cauchy-Schwartz inequality and the embedding $U^2\subset V^2_{rc},$  we have
\begin{align*}
   & \sum_k \sum_{k_1,k_2,k_3=k-10}^{k+9} \left|\int_{0}^{T} \int_{\R^2}f_{k_1,k_2,k_3,k}dxdt\right| \\
   & \lesssim \left( \sum_{k} \|P_{k}u_1\|_{U^2_{h(D)}}^2\right)^{1/2} \left( \sum_{k}   \|P_{k}u_2\|_{U^2_{h(D)}}^2 \right)^{1/2}
   \left( \sum_{k} \|P_{k}u_3\|_{U^2_{h(D)}}^2\right)^{1/2} \left( \sum_{k}   \|P_{k}v\|_{V^2_{h(D)}}^2 \right)^{1/2} \\
   & \lesssim \|u_1\|_{X^0} \|u_2\|_{X^0} \|u_3\|_{X^0} \|v\|_{Y^0} \\
   & \lesssim \|u_1\|_{X^0} \|u_2\|_{X^0} \|u_3\|_{X^0}.
\end{align*}

Now we consider the off-diagonal sum, that is, the sum over $\tilde{k}_1\geq k+10.$ We will only consider the sum over $k_1=\tilde{k_1}\geq k+10,$ since the remaining sum can be split into finitely many similar cases and be similarly treated.
Applying H\"{o}lder's inequality, Corollary \ref{bilinearestimate3V2} once and Corollary \ref{linearestimateV2} twice, we obtain when $k+10\leq  k_1\leq M$
\begin{align*}
  & \left|\int_{0}^{T}\int_{\R^2} f_{k_1; k-10 \leq k_2,k_3\leq M;k}dxdt \right|  \\
  & \qquad \qquad \qquad \qquad = \left|\int_{0}^{T} \int_{\R^2}P_{k_1}u_1 \overline{P_{k-10 \leq k_2\leq M}u_2} P_{k-10 \leq k_3\leq M}u_3 \overline{P_kA(D) v}dxdt\right|  \\
  & \qquad \qquad \qquad \qquad \lesssim \|P_{k_1}u_1 \overline{P_kA(D)v}\|_{L^2_{t,x}} \|P_{k-10 \leq k_2\leq M}u_2\|_{L^4_{t,x}}\|P_{k-10 \leq k_3\leq M}u_3\|_{L^4_{t,x}}   \\
  & \qquad \qquad \qquad \qquad \lesssim 2^{\beta k/2} 2^{-\beta k_1/4} 2^{-2\beta k/8} \|P_{k_1}u_1\|_{V^2_{h(D)}} \|P_kv\|_{V^2_{h(D)}} \|u_2\|_{V^2_{h(D)}} \|u_3\|_{V^2_{h(D)}}  \\
  & \qquad \qquad \qquad \qquad \lesssim 2^{(k-k_1)\beta/4} \|P_{k_1}u_1\|_{U^2_{h(D)}} \|P_kv\|_{V^2_{h(D)}} \|u_2\|_{X^0} \|u_3\|_{X^0}.
\end{align*}
Here we have also used the embedding (\ref{embeddingX0Y0}). Therefore by the Cauchy-Schwartz inequality we conclude that
\begin{align*}
  & \sum_{k\leq M} \sum_{k+10\leq k_1\leq M} \left|\int_{0}^{T}\int_{\R^2} f_{k_1; k-10 \leq k_2,k_3\leq M;k}dxdt \right| \\
  & \qquad \qquad \qquad \qquad \qquad \lesssim \|u_2\|_{X^0} \|u_3\|_{X^0}
  \left(\sum_{k\leq M} \sum_{k+10\leq k_1\leq M} 2^{\beta(k-k_1)/4} \|P_{k_1}u_1\|^2_{U^2_{h(D)}} \right)^{1/2}  \\
  & \qquad \qquad \qquad \qquad  \qquad \qquad \times \left(\sum_{k\leq M} \sum_{k+10\leq k_1\leq M} 2^{\beta(k-k_1)/4}  \|P_{k}v\|^2_{V^2_{h(D)}} \right)^{1/2}  \\
   & \qquad \qquad  \qquad \qquad \qquad \lesssim \|u_1\|_{X^0} \|u_2\|_{X^0} \|u_3\|_{X^0} \|v\|_{Y^0} \\
   & \qquad \qquad  \qquad \qquad \qquad \lesssim \|u_1\|_{X^0} \|u_2\|_{X^0} \|u_3\|_{X^0}.
\end{align*}

\medbreak

\noindent
{\bf Estimate of $I_2.$}
We will only consider the sum over $k_3=\tilde{k_3}< k-10.$
Applying H\"{o}lder's inequality, Corollary \ref{linearestimateV2} twice and Corollary \ref{bilinearestimate3V2} once, we obtain when $k_3<  k-10,$
\begin{align*}
  & \left|\int_{0}^{T}\int_{\R^2} f_{k-10 \leq k_1, k_2\leq M; k_3; k}dxdt \right| \\
  & \qquad \qquad \qquad \qquad \qquad  \lesssim \|P_{k_3}u_3 P_{k}A(D)v \|_{L^2_{t,x}}\|P_{k-10\leq k_1\leq M}u_1\|_{L^4_{t,x}} \|P_{k-10\leq \cdot \leq M}u_2\|_{L^4_{t,x}}  \\
  & \qquad \qquad \qquad \qquad \qquad  \lesssim 2^{-\beta k_3/4}2^{-2\beta k/8} 2^{\beta k/2}  \|u_1\|_{X^0} \|u_2\|_{X^0} \|P_{k_3}u_3\|_{V^2_{h(D)}} \|P_kv\|_{V^2_{h(D)}} \\
  & \qquad \qquad \qquad \qquad \qquad  \lesssim 2^{\beta(k-k_3)/4}  \|P_{k_3}u_3\|_{U^2_{h(D)}} \|P_kv\|_{V^2_{h(D)}} \|u_1\|_{X^0} \|u_2\|_{X^0},
\end{align*}
Here  we have again used the embedding (\ref{embeddingX0Y0}). Now we sum the above estimate over $k\leq M, k_3< k-10$ using the Cauchy-Schwartz inequality as before to conclude that
$$I_2\lesssim \|u_1\|_{X^0}\|u_2\|_{X^0}\|u_3\|_{X^0}.$$

\medbreak
\noindent
{\bf Estimate of $I_3.$}
We will only consider the sum over $k_2=\tilde{k_2}< k-10,$ $ k_3=\tilde{k_3}< k-10.$
Applying H\"{o}lder's inequality and Corollary \ref{linearestimateV2}, \ref{bilinearestimate3V2}, we obtain when $k-10\leq k_1 \leq M,$
\begin{align*}
  & \left|\int_{0}^{T}\int_{\R^2} f_{k_1; k_2< k-10; k_3< k-10; k}dxdt \right| \\
  & \qquad \lesssim \left(\|P_{k_1}u_1 \overline{P_{<k-20}u_2}\|_{L^2_{t,x}}+\|P_{k_1}u_1 \overline{P_{k-20\leq \cdot <k-10}u_2}\|_{L^2_{t,x}}\right) \|P_{<k-10}u_3\overline{P_{k}A(D)v}\|_{L^2_{t,x}}  \\
  & \qquad \lesssim \left(2^{-\beta(k+k_1)/4}+2^{-\beta(k_1+3k)/8}\right)2^{\beta k/2}   \|u_2\|_{X^0} \|u_3\|_{X^0} \|P_{k_1}u_1\|_{V^2_{h(D)}} \|P_kv\|_{V^2_{h(D)}} \\
  & \qquad \lesssim 2^{\beta(k-k_1)/8}  \|P_{k_1}u_1\|_{U^2_{h(D)}} \|P_kv\|_{V^2_{h(D)}} \|u_2\|_{X^0} \|u_3\|_{X^0}.
\end{align*}
We sum the above estimate over $k-10\leq k_1\leq M,$ and use the Cauchy-Schwartz inequality as before to conclude
$$\sum_{k\leq M} \sum_{k-10\leq k_1\leq M}\left|\int_{0}^{T}\int_{\R^2} f_{k_1; k_2<k-10; k_3< k-10; k}dxdt \right| \leq \|u_1\|_{X^0}\|u_2\|_{X^0}\|u_3\|_{X^0}.$$
Hence we obtain
$$I_3\lesssim \|u_1\|_{X^0}\|u_2\|_{X^0}\|u_3\|_{X^0}.$$

\medbreak
\noindent
{\bf Estimate of $I_4.$}
Without loss of generality we only consider the sum over $k_1\leq k_2\leq k_3<k-10.$
We let $N>0$ be an admissible integer which will be chosen later. We let $T_i$ be the Fourier multiplier with symbol $1_{2\pi i/N\leq \xi/|\xi|<2\pi(i+1)/N}(\xi)$ where we again use the identification $\mathbb{S}^{1}\cong \R/2\pi \Z.$
Then
$$\int_{0}^{T}\int_{\R^2} f_{k_1,k_2,k_3, k}dxdt=\sum_{i_1,i_2,i_3,i=0}^{N-1} \int_{0}^{T}\int_{\R^2} (T_{i_1}P_{k_1}u_1) (\overline{T_{i_2}P_{k_2}}) (T_{i_3}P_{k_3}u_3)(\overline{T_{i}P_{k}A(D)v})dxdt.$$

We let $O_{i,k}$ be the open set
$$O_{i,k}=\{\xi \in \R^2: 2\pi (i-1/3)/N< \xi/|\xi|<2\pi(i+4/3)/N, \quad 2^{k-2}<\left||\xi|-1\right| <2^{k+2}\}$$
By Plancherel's theorem,  in order for the integral
$$\int_{0}^{T}\int_{\R^2} (T_{i_1}P_{k_1}u_1) (\overline{T_{i_2}P_{k_2}u_2}) (T_{i_3}P_{k_3}u_3)(\overline{T_{i}P_{k}A(D)v})dxdt$$
not to vanish,
we must have
\begin{equation}\label{e1}
  \{\xi_1+\xi_3:\xi_1\in O_{i_1,k_1}, \xi_3\in O_{i_3,k_3} \}  \bigcap
  \{\xi_2+\xi: \xi_2\in O_{i_2,k_2}, \xi\in O_{i,k} \}\neq \emptyset.
\end{equation}
When $k_1,k_2,k_3<k-10,$ (\ref{e1}) implies that there exists an admissible integer $N>0$ satisfying $N\sim 2^{k/2}$ and an admissible constant $d>0$ such that
\begin{multline}\label{e2}
  \int_{0}^{T}\int_{\R^2} f_{k_1,k_2,k_3, k}dxdt=\sum_{i_2,i=0}^{N-1} \left( \sum_{|i_1-i_2|\leq d} \sum_{|i_3-i|\leq d} +\sum_{|i_1-i|\leq d} \sum_{|i_3-i_2|\leq d} \right) \\
   \int_{0}^{T}\int_{\R^2} (T_{i_1}P_{k_1}u_1) (\overline{T_{i_2}P_{k_2}u_2}) (T_{i_3}P_{k_3}u_3)(\overline{T_{i}P_{k}A(D)v})dxdt.
\end{multline}
In particular (\ref{e1}) implies that if we let $\theta_1, \theta_2$ be the angles between $\xi_1,\xi_3$ and $\xi_2,\xi$ respectively, then
$$2(\cos \theta_1-\cos \theta_2)|\xi_1||\xi_3|=|\xi_1|+|\xi_3|-|\xi_2|-|\xi|-2\cos \theta_2 (|\xi_1||\xi_3|-|\xi_2||\xi|).$$
The above constraint implies that if we choose $N$ large enough (still satisfying $N\sim 2^{k/2}$) then
$$\int_{0}^{T}\int_{\R^2} (T_{i_1}P_{k_1}u_1) (\overline{T_{i_2}P_{k_2}u_2}) (T_{i_3}P_{k_3}u_3)(\overline{T_{i}P_{k}A(D)v})dxdt=0$$
if $|i_1-i_3|\leq 2,$ $N/2-2\leq |i_1-i_2|\leq N/2+2,$ $N/2-2\leq |i_3-i_2|\leq N/2+2$ all hold.

We let $l,\theta$ be the parameters in Proposition \ref{bilinearestimateoriginal} for $O_{i',k'},O_{i'',k''}$ with $k'\leq k'',$ $k'<k-10$ and $2<|i'-i''|\leq d.$ Then the same argument in the proof of Proposition \ref{bilinearestimatefreewave} combined with Lemma \ref{geometriclemma} shows that
$$l\lesssim 2^{k'}+2^{(k'-k'')/2}\lesssim 2^{(k'-k'')/2},\quad \theta \gtrsim 2^{k/2}.$$
So Proposition (\ref{bilinearestimateoriginal}) implies that
$$\|(e^{ith(D)}T_{i'}P_{k'}u)(e^{ith(D)}T_{i''}P_{k''}v)\|_{L^2_{t,x}}\lesssim 2^{(k'-k'')/4} 2^{-k/4} \|T_{i'}P_{k'}u\|_{L^2_x} \|T_{i''}P_{k''}v\|_{L^2_x}$$
when $2<|i'-i''|\leq d.$
Similarly we have
$$\|(e^{ith(D)}T_{i'}P_{k'}u)(\overline{e^{ith(D)}T_{i''}P_{k''}v})\|_{L^2_{t,x}}
\lesssim 2^{(k'-k'')/4} 2^{-k/4} \|T_{i'}P_{k'}u\|_{L^2_x} \|T_{i''}P_{k''}v\|_{L^2_x}.$$
when $-N/2+2<|i'-i''|<N/2-2.$
In particular $2^{(k'-k'')/4}2^{-k/4}\leq 2^{(k'-k)/4} 2^{-\beta k/4}.$
Therefore arguing as in the proof of (\ref{bilinearestimatefreewave}) and the argument deriving Corollary \ref{bilinearestimate3V2}, we conclude from the above analysis that when $k_1\leq k_2\leq k_3<k-10,$
\begin{align*}
  & \left|\int_{0}^{T}\int_{\R^2} (P_{k_1}u_1) (\overline{P_{k_2}u_2}) (P_{k_3}u_3)(\overline{P_{k}A(D)v})dxdt\right| \\ & \qquad \qquad \qquad \qquad
  \lesssim  \left( 2^{(k_1-k)/8} 2^{(k_2-k_3)/8} + 2^{(k_2-k)/8} 2^{(k_1-k_3)/8}+2^{(k_3-k)/8} 2^{(k_1-k_2)/8} \right)  \\ &\quad \qquad \qquad \qquad \qquad
  \times 2^{-\beta k/2} 2^{\beta k/2} \|P_{k_1}u_1\|_{V^2_{h(D)}} \|P_{k_2}u_2\|_{V^2_{h(D)}}\|P_{k_3}u_3\|_{V^2_{h(D)}}
   \|P_{k}v\|_{V^2_{h(D)}}.
\end{align*}
Summing the above estimate over $k_1\leq k_2\leq k_3<k-10$ and using the Cauchy-Schwartz inequality yield
$$I_4\lesssim \|u_1\|_{X^0}\|u_2\|_{X^0}\|u_3\|_{X^0}.$$

\section{Proof of the Main Theorem}\label{proofmainsection}

We recall that  $I_T$ is the linear operator given by
\begin{equation*}
  I_T(f)(t,x)=\int_{0}^{t}1_{[0,T)}(s)e^{i(t-s)h(D)}f(s)ds
\end{equation*}
where $T\in [0,\infty].$
By Duhamel's formula we know that $u\in C^0([0,T),L^2_x)$ is a strong solution to the Cauchy problem (\ref{modeleqn}) if and only if it satisfies the following integral equation
$$u(t)=e^{ith(D)} u_0 +i\int_{0}^{t}e^{i(t-s)h(D)}A(D)\left(|P_{\leq M}u(s)|^2 P_{\leq M}u(s)\right)ds$$
on $[0,T),$ which is implied by
\begin{equation}\label{duhammel}
  u(t)=1_{[0,T)}(t)e^{ith(D)}u_0+1_{[0,T)}(t)I_{T}\left(A(D)\left(|P_{\leq M}u|^2 P_{\leq M}u\right)\right).
\end{equation}
The identity (\ref{duhammel}) also implies that $u(t)=0$ outside $[0,T).$

We let $N_T$ be the (nonlinear) operator given by
\begin{equation}\label{Noperator}
  N_{T}(u)=1_{[0,T)}(t)I_{T}\left(A(D)\left(|P_{\leq M}u|^2 P_{\leq M}u\right)\right).
\end{equation}
We will use a fixed point argument to the map\footnote{Note that the map here depends on $u_0,$ although not indicated explicitly in the notation.}
$$\phi_T: u\mapsto 1_{[0,T)}(t)e^{ith(D)}u_0 +N_{T}(u)$$
to obtain a strong solution $u\in C^0_t([0,T),L^2_x).$ Such fixed point argument of establishing wellposedness results has been a standard method (see for example \cite{Tsutsumi} \cite{tao2006nonlinear}).

\begin{proof}[Proof of  Theorem \ref{smallglobalwellposed}]
   Let $\e_0 ,\delta_0>0$ be small constants to be chosen later, and let  $T=\infty.$ We let $B_{L^2_x}(0,\e_0)$ be the closed ball of radius $\e_0$ in the space $L^2_x(\R^2),$ and similarly we let $B_{X^0}(0,\delta_0)$ be the closed ball of radius $\delta_0$ in $X^0.$ Because of the embedding $X^0\hookrightarrow L^\infty_tL^2_x,$ $B_{X^0}(0,\delta_0)\cap C^0([0,\infty),L^2_x)$ is a closed subspace of $X^0,$ and we put $X^0-$norm on it which turns it into a Banach space.\footnote{By $B_{X^0}(0,\delta_0)\cap C^0([0,\infty),L^2_x)$ we mean the space of functions $u\in B_{X^0}(0,\delta_0)$ such that $u$ is continuous on $[0,\infty)$ as a function into $L^2_x.$}

   Since the nonlinearity is of algebraic power type, we can show using $\|1_{[0,\infty)}\cdot\|_{X^0}\lesssim \|\cdot\|_{X^0},$ Proposition \ref{nonlinearestimate} and the triangle inequality that
   $$\|N_\infty (u)-N_\infty (v)\|_{X^0}\leq c_1 (\|u\|_{X^0}^2+\|v\|_{X^0}^2)\|u-v\|_{X^0},$$
   for some admissible constant $c_1>0.$
   We also have
   \begin{equation}\label{7}
     \|\phi_\infty(u)\|_{X^0}\leq c_2\left( \|u_0\|_{L^2_x}+\|u\|_{X^0}^3 \right)
   \end{equation}
   for some admissible constant $c_2>0,$
   since $\|1_{[0,\infty)}(t)e^{ith(D)}u_0\|_{X^0}\lesssim \|1_{[0,\infty)} u_0\|_{U^2}\lesssim \|u_0\|_{L^2_x}.$
   Therefore we can choose $\delta_0, \e_0>0$ small enough depending only on $c_1,c_2$ such that when $u_0\in B_{L^2_x}(0,\e_0),$
   $$\phi_\infty: B_{X^0}(0,\delta_0)\cap C^0([0,\infty),L^2_x) \rightarrow B_{X^0}(0,\delta_0)\cap C^0([0,\infty),L^2_x)$$
   and we have a strict contraction
   $$\|\phi_\infty (u)-\phi_\infty (v)\|_{X^0}\leq \frac{1}{2}\|u-v\|_{X^0}.$$
   Therefore for every $u_0\in B_{L^2_x}(0,\e_0),$ there exists a unique fixed point $u\in B_{X^0}(0,\delta_0)\cap C^0([0,\infty),L^2_x)$ of $\phi_\infty,$ which by Duhamel's formula is a strong global solution to the Cauchy problem (\ref{modeleqn}).

   For fixed points $u,v$ of $\phi_\infty$ with initial data $u_0,v_0\in B_{L^2_x}(0,\e_0)$ respectively, we have
   $$\|u-v\|_{X^0}\leq c_2\|u_0-v_0\|_{L^2_x}+2c_1\delta_0^2 \|u-v\|_{X^0}.$$
   If we choose $\delta_0,\e_0$ sufficiently small then we can conclude the map
   $$B_{L^2_x}(0,\e_0)\rightarrow B_{X^0}(0,\delta_0)\cap C^0([0,\infty),L^2_x), \, u_0\mapsto u$$
   is Lipschitz continuous. Hence the solution map
   $$B_{L^2_x}(0,\e_0)\rightarrow B_{X^0([0,\infty))}(0,\delta_0)\cap C^0([0,\infty),L^2_x),\, u_0\mapsto u|_{[0,\infty)}$$
   is Lipschitz continuous. Scattering of solutions is an immediate consequence of the fact that the limit $\lim_{t\rightarrow \infty} g$ exists for any function $g\in U^2.$

   To establish unconditional uniqueness of solutions in $X^0([0,\infty))\cap C^0([0,\infty),L^2_x),$ we will show that if $u,v\in X^0([0,\infty))\cap C^0([0,\infty),L^2_x)$ are two strong solutions to our Cauchy problem (\ref{modeleqn}) with $u(0)=v(0),$ then $u=v$ on $[0,\infty).$ By the time-translation invariance of our Cauchy problem (\ref{modeleqn}) and a continuity argument, it suffices to show that if $u(0)=v(0),$ then the solutions $u,v$ agree on a short time interval $[0,T')$ for some $T'>0.$ However this is immediate from local wellposedness of our Cauchy problem (\ref{modeleqn}).
\end{proof}

%
%

\bibliographystyle{alpha}
\bibliography{Fu-Tataru}

\begin{thebibliography}{CDMM90}

\bibitem[Bou98]{bourgain1998refinements}
Jean Bourgain.
\newblock Refinements of {Strichartz} inequality and applications to {2D-NLS}
  with critical nonlinearity.
\newblock {\em International Mathematics Research Notices}, 1998(5):253--283,
  1998.

\bibitem[CDMM90]{cowling1990damping}
Michael Cowling, Shaun Disney, Giancarlo Mauceri, and Detlef M{\"u}ller.
\newblock Damping oscillatory integrals.
\newblock {\em Inventiones mathematicae}, 101(1):237--260, 1990.

\bibitem[CH16]{herrcandy2016transference}
Timothy Candy and Sebastian Herr.
\newblock Transference of bilinear restriction estimates to quadratic variation
  norms and the {Dirac-Klein-Gordon} system.
\newblock {\em arXiv preprint arXiv:1605.04882}, 2016.

\bibitem[CKZ07]{carbery2007restriction}
Anthony Carbery, Carlos Kenig, and Sarah Ziesler.
\newblock Restriction for flat surfaces of revolution in �$\mathbb{R}^3$.
\newblock {\em Proceedings of the American Mathematical Society},
  135(6):1905--1914, 2007.

\bibitem[HHK09]{hadac2009well}
Martin Hadac, Sebastian Herr, and Herbert Koch.
\newblock Well-posedness and scattering for the {KP-II} equation in a critical
  space.
\newblock {\em Annales de l'Institut Henri Poincare (C) Nonlinear Analysis},
  26(3):917--941, 2009.

\bibitem[KPV91]{kenig1991oscillatory}
Carlos Kenig, Gustavo Ponce, and Luis Vega.
\newblock Oscillatory integrals and regularity of dispersive equations.
\newblock {\em Indiana University Mathematics Journal}, 40(1):33--69, 1991.

\bibitem[KT05]{kochtataru2005upvp}
Herbert Koch and Daniel Tataru.
\newblock Dispersive estimates for principally normal pseudodifferential
  operators.
\newblock {\em Communications on pure and applied mathematics}, 58(2):217--284,
  2005.

\bibitem[KTV14]{kochtataruvisan2014dispersive}
Herbert Koch, Daniel Tataru, and Monica Visan.
\newblock {\em Dispersive Equations and Nonlinear Waves: Generalized
  Korteweg--de Vries, Nonlinear Schr{\"o}dinger, Wave and Schr{\"o}dinger
  Maps}, volume~45 of {\em Oberwolfach Seminars}.
\newblock 2014.

\bibitem[Lan13]{lannes2013water}
David Lannes.
\newblock {\em The water waves problem}, volume 188 of {\em Mathematical
  surveys and monographs}.
\newblock Americal Mathematical Society, 2013.

\bibitem[Obe04]{oberlin2004uniform}
Daniel Oberlin.
\newblock A uniform {Fourier} restriction theorem for surfaces in
  $\mathbb{R}^3$.
\newblock {\em Proceedings of the American Mathematical Society}, pages
  1195--1199, 2004.

\bibitem[Sha07]{shayya2007affine}
Bassam Shayya.
\newblock An affine restriction estimate in $\mathbb{R}^3$.
\newblock {\em Proceedings of the American Mathematical Society},
  135(4):1107--1113, 2007.

\bibitem[SM93]{stein1993harmonic}
Elias Stein and Timothy Murphy.
\newblock {\em Harmonic Analysis: Real-Variable Methods, Orthogonality, and
  Oscillatory Integrals}, volume~43 of {\em PMS}.
\newblock 1993.

\bibitem[SS85]{sogge1985averages}
Christopher Sogge and Elias Stein.
\newblock Averages of functions over hypersurfaces in $\mathbb{R}^n$.
\newblock {\em Inventiones mathematicae}, 82(3):543--556, 1985.

\bibitem[Tao06]{tao2006nonlinear}
Terence Tao.
\newblock {\em Nonlinear dispersive equations: local and global analysis},
  volume 106 of {\em CBMS}.
\newblock 2006.

\bibitem[Tsu87]{Tsutsumi}
Y.~Tsutsumi.
\newblock {$L^2$} solutions for nonlinear {Schr\"{o}dinger} equations and
  nonlinear groups.
\newblock {\em Funk. Ekva.}, 30:115--125, 1987.

\end{thebibliography}

\end{document}